\documentclass[11pt]{amsart}
\usepackage{amsmath,amssymb, graphicx, amscd,latexsym,color,comment}

\makeatletter
\newtheorem{Theorem}{Theorem}
\newtheorem{Lemma}[Theorem]{Lemma}
\newtheorem{Corollary}[Theorem]{Corollary}
\newtheorem{Proposition}[Theorem]{Proposition}
\newtheorem{Sublemma}[Theorem]{Sublemma}

\newtheorem{Remark}[Theorem]{Remark}

\newtheorem{Assertion}[Theorem]{Assertion}

\newtheorem{Problem}[Theorem]{Problem}

\newcommand{\eps}{\varepsilon}

\newcommand\la{\lambda}

\newcommand\al{\alpha}

\newcommand\si{\sigma}
\newcommand\be{\beta}
\newcommand\Si{\Sigma}
\newcommand\ga{\gamma}
\newcommand\Ga{\Gamma}
\newcommand\de{\delta}
\newcommand\De{\Delta}

\newcommand\bfz{\mbox {\bf  z}}

\newcommand\Int{\rm{Int}\/}

\newcommand\Cone{\rm{Cone}\/}

\newcommand\Vol{\rm{Vol}\/}

\newcommand\inv{^{-1}}

\def\mapright#1{\smash{\mathop{\longrightarrow}\limits^{{#1}}}}

\def\inv{^{-1}}

\begin{document}

\title{Almost non-degenerate functions and a Zariski pair of links}

\author
[M. Oka ]
{Mutsuo Oka \\\\
{\tiny Dedicated to Professor Norbert A'Campo on his 80th birthday}}
\address{
Department of Mathematics, 
Tokyo  University of Science,
Kagurazaka 1-3, Shinjuku-ku,
Tokyo 162-8601}
 \email{oka@rs.tus.ac.jp}


\begin{abstract}
Let $f(\bfz)$ be an analytic function defined in the neighborhood of  the origin of $\mathbb C^n$ which have some Newton degenerate faces.
We generalize the  Varchenko  formula for the zeta function of the Milnor fibration  of a Newton non-degenerate function $f$ to this case.
As an application, we give an example of a pair of hypersurfaces with the same Newton boundary and the same zeta function with different  tangent cones.
\end{abstract}
\maketitle

\maketitle
\setcounter{tocdepth}{1}
\tableofcontents
 
\section{Introduction}
Consider an analytic function $f(\mathbf z)=\sum_{\nu}a_\nu\mathbf z^{\nu}$ of $n$ variables defined in a neighborhood of the origin of $\mathbb C^n$. 
Assume that we are given a good resolution $ \hat\pi:X\to \mathbb C^n$ of the function $f$ and let $E_1,\dots, E_s$ be the exceptional divisors of $ \hat\pi$,
that is    $ \hat\pi^*f\inv(V)= \widetilde V\cup_{i=1}^s E_j$ where $\widetilde V$ is the strict transform of the hypersurface $V=f\inv(0)$. 
Consider the open dense subset $E_j''=E_j\cap { \hat\pi}\inv(0)\setminus \widetilde V\cup_{i\ne j}E_i$. 
Let $m_j$ be the multiplicity of $ \hat\pi^*f$ along $E_j$. By A'Campo \cite{ACampo},
the zeta function of the Milnor monodromy at the origin is given as
\begin{eqnarray}(AC)\label{AC}
\hspace{2cm}\zeta(t)=\prod_{j=1}^s(1-t^{m_j})^{-\chi(E_j'')}.\hspace{2cm}\notag
\end{eqnarray}
Suppose that $f(\mathbf z)$ is Newton non-degenerate. Then using a toric modification
$ \hat\pi:X\to \mathbb C^n$ which is  admissible with the dual Newton diagram $\Ga^*(f)$, the zeta function can be 
computed combinatorially (Varchenko, \cite{Va}).  More precisely 
the zeta function is given  as 
\begin{eqnarray}(V)\label{Va}
\quad \zeta(t)&=\prod_{I}\zeta_I(t),\quad \zeta_I(t)=\prod_{Q\in \mathcal P_I}(1-t^{d(Q;f^I)})^{-\chi(Q)}\notag
\end{eqnarray}
where  $f^I$ is the restriction of $f$ to the coordinate subspace $\mathbb C^I$ and  $d(Q,f^I)$ is the minimal value of the monomials in $f^I$ with respect to the weight vector $Q$. 
$I\subset \{1,\dots,n\}$ moves so that $f^I$ is not zero.   $\mathcal P_I$ is the set of primitive integer weight vectors of 
the coordinate subspace $\mathbb C^I$
 which correspond to the maximal faces of $\Ga(f^I)$.  For $I=\{1,\dots, n\}$, we simply denote $\mathcal P$.
The number $\chi(Q)$ in (V) is defined as follows.
\begin{eqnarray}\label{Var}
\begin{split}
 \chi(Q)&=(-1)^{|I|-1}|I|!{\Vol}_{|I|}C(\De(Q;f^I),O_I)/d(Q;f^I).\notag
\end{split}
\end{eqnarray}
Here $C(\De(Q;f^I),O_I)$ is the cone over $\De(Q;f^I)$ with a vertex at the origin $O_I$ of $\mathbb C^I$
and $\Vol_k$ is the $k$-dimensional Euclidean volume. See \cite{Va,Ok12} or Chapter 3 of \cite{Okabook}
for details.
The formula (V) does not require any explicit regular subdivision of $\Ga^*({f})$ and is convenient for computing the zeta function. However,  to compute the intersection numbers of 
exceptional divisors,
 we need to use  an explicit toric resolution.   
 A'Campo's formula using a toric modification is convenient in such a situation, as it express the geometry more directly.
  
 The purpose of this paper is to generalize Varchenko's formula for   certain  functions which have some Newton  degenerate faces. 
 In \S 2, we recall the basic definitions about good resolutions, the Newton boundary and the non-degeneracy and  admissible toric modifications with respect to the dual Newton diagram.
In \S 3, we introduce the class of almost non-degenerate functions and give the first main result on the zeta function (Theorem \ref{main1}).
  As an application
 of Theorem \ref{main1},   we give a Zariski pair of links. Namely we give two hypersurfaces of dimension 2 with the same zeta function 
whose tangent cones gives a Zariski pair in $\mathbb P^{2}$ in \S 4 (Theorem \ref{main2}, Theorem \ref{main3}).
 \begin{Remark}{\rm 
 There is a canonical projection $\pi_I: \mathbb Z^n\to {\mathbb Z}^I$ 
 associated with the projection $\pi^I:\mathbb R^n\to\mathbb R^I$ but 
 $P_I=\pi_I(P)$  is not necessarily  primitive  for  $P\in \mathcal P$.
  In the formula (V), $Q\in \mathcal P_I$ is not necessarily   a vertex of $\Si^*$ and  in general $Q$ is  not in the image $\pi_I(\mathcal P)$ but we  use the information $Q\in \mathcal P_I$ for the calculation. See  the proof of Theorem (5.3) (\cite{Okabook}).
 On the other hand,  using an admissible toric modification, (AC) is restated as 
 \[
(\rm{AC'})\hspace{2cm} \zeta(t)=\prod_{P\in \mathcal V^+}(1-t^{d(P)})^{-\hat E(P)''}
\hspace{3cm}
 \]
 where $\mathcal V^+$ is the set of strictly positive vertices of $\Si^*$, which do  not necessarily
 correspond to the maximal dimensional faces. In general, $\mathcal P\subset \mathcal V^+$.
 Thus we need only the information of the vertices in $\mathcal V^+$ but we do not use lower faces $f^I$.
}\end{Remark}
In this paper we use the notations:
\[\begin{split}
\mathbb C^I&=\{\mathbf z\in \mathbb C^n\,|\, z_j=0\, j\not\in I\}, \, f^I=f|_{\mathbb C^I},\\
\mathbb C^{*I}&=\{\mathbf z\in \mathbb C^n\,|\, z_j=0\,\iff  j\not\in I\}.
\end{split}
\]
In particular, we write simply $\mathbb C^n$ and $\mathbb C^{*n}$ for $I=\{1,\dots,n\}$.
 \section{Preliminaries}
 \subsection{A good resolution of a function} Let $f$ be an analytic function defined in a neighborhood $U$ of the origin of $\mathbb C^n$.
 Let $X$ be a complex manifold of dimension $n$ and 
 $ \hat\pi:X\to  U$ is a  proper holomorphic function. $ \hat\pi:X\to U$ is called {\em a good resolution of $f$} if it satisfies the following:
 \begin{enumerate}
  \item  $ \hat\pi$ is biholomorphic  on the restriction to
 $X\setminus  \hat\pi\inv(V)\to U\setminus V,\, V= f\inv (0)$. 
 
 Assume that
 the divisor $ (\hat\pi^*f)$ is given as  $\widetilde V+\sum_{i=1}^k m_i E_i$ where $\widetilde V$ is the strict transform of $V=f\inv(0)$ and $m_i$ is the multiplicity of $\hat\pi^*f$ along $E_i$. Let $\widetilde V_i,\,i=1,\dots, m$  be the irreducible components of $\widetilde V$.
 \item Each irreducible component $\widetilde V_i$ and the divisors $E_1,\dots, E_k$ are non-singular  and $\widetilde V\cup_{i=1}^k E_i$ has, at most, ordinary normal crossing
 singularities. Namely take  $p\in \hat\pi\inv(0)$ and  let $I\subset \{1,\dots,k+m\}$ be the set of $i$ such that 
 $p\in  E_i$. Then $|I|\le n$ and there is an analytic coordinate  chart$(U_p,(v_1,\dots, v_n))$ in a small neighborhood 
 $U_p$ of $p$  and an injective mapping $\xi:I\to \{1,\dots,n\}$ 
  so that $E_i=\{v_{\xi(i)}=0\}$ for $i\in I$. Here we put $E_{k+i}=\widetilde V_i$ for simplicity.
  \end{enumerate}
  \subsection{The Newton boundary and the dual Newton diagram}
  Let $M$ be the space of monomials of the fixed coordinate variables $z_1,\dots, z_n$ of $\mathbb C^n$ and let $N$ be the space of weights  of the  variables $z_1,\dots,z_n$. We identify   the monomial $\mathbf z^{\nu}=z_1^{\nu_{\si 1}}\dots z_n^{\nu_n}$ and the integral point $\nu=(\nu_{\si 1},\dots, \nu_n)\in \mathbb R^n$. A weight $P$  is also identified with
  the column vector ${}^t(p_1,\dots, p_n)\in \mathbb R^n$ where  $p_i=\deg_P (z_i)$ and we call $P$ a weight vector.
  Let $f(\mathbf z)=\sum_\nu a_\nu \mathbf z^\nu$ be a given  holomorphic function defined by a convergent series. 
The  Newton polygon $\Ga^+(f)$ with respect to the given coordinates $\mathbf z=(z_1,\dots, z_n)$ is the convex hull of the union $\cup_{\nu,a_\nu\ne 0} \{\nu+\mathbb R_{\ge 0}^n\}$ and {\em the Newton boundary} $\Ga(f)$  is defined  by the union of compact faces of $\Ga^+(f)$.
An integral point $\nu=(\nu_{\si 1},\dots, \nu_n)\in \Ga^+(f)$ corresponds to the monomial $\mathbf z^\nu=z_1^{\nu_{\si 1}}\dots z_n^{\nu_n}$ and  we consider  $\Ga^+(f), \Ga(f)$ as subspaces of $M^+_{\mathbb R}$ where $M^+_{\mathbb R}=M\otimes \mathbb R$
and we identify $M^+_{\mathbb R}$ with $ \mathbb R_{\ge 0}^n$. Similarly we identify $N_{\mathbb R}:=N\otimes \mathbb R$ with
$\mathbb R^n$.  
  For a positive   weight vector $P={}^t(p_1,\dots,p_n)$, we consider the canonical linear function $\ell_P$ on $\Ga^+(f)$
  which is defined by $\ell_P(\nu)=\sum_{i=1}^n \nu_ip_i$.
  This is nothing but  the degree mapping $\deg_P \mathbf z^\nu=\sum_{i=1}^n p_i\nu_i$.
   The minimal value of $\ell_P$ is denoted by
  $d(P;f)$.  Put $\De(P;f):=\{\nu\in \Ga^+(f)\,|\, \ell_P(\nu)=d(P)\}$. We will  use the simplified notations $d(P)$ and $\De(P)$ if any ambiguity seems unlikely.  In general,
  $\De(P)$ is a face of $\Ga^+(f)$ and $\De(P)\subset \Ga(f)$ if $P$ is {\it strictly positive} (i.e., $p_i>0,\forall i$).
A face   $\De\subset \Ga(f)$ with $\dim\, \De=n-1$, there is a unique strictly positive primitive integer vector $P$ such that $\De(P)=\De$. (Recall that $P={}^t(p_1,\dots, p_n)$ is primitive if $\gcd\,\{p_1,\dots,p_n\}=1$.)
  The  partial sum $\sum_{\nu\in \De}a_\nu \mathbf z^\nu$ is called {\it the face function} of weight $P$ and we denote it  as 
  $f_P$ or $f_{\De}$. It is a polynomial if $P$ is strictly positive.
  Two weight vectors $P,Q$ are equivalent if and only if $\De(P)=\De(Q)$ and this equivalent relation gives a conical subdivision of the positive weight vectors $N^+_{\mathbb R}$, i.e.  of $\mathbb R_{\ge 0}^n$
  (under the above identification) and we denote it as $\Ga^*(f)$ and call it {\em the dual Newton diagram of $f$}. 
  We say, $f$ is {\it Newton non-degenerate
  on a face  $\De$} of $\Ga(f)$ if $f_\De:\mathbb C^{*n}\to \mathbb C$ has no critical points. $f$ is {\it Newton non-degenerate} if it is non-degenerate on every face $\De\subset \Ga(f)$ of any dimension. 
  The closure of an  equivalent class can be irredundantly expressed  as 
  \[
  {\Cone}(P_1,\dots, P_k):=\left\{\sum \la_i P_i\,|\, \la_i\ge 0\right \}\]
  where  $P_1,\dots, P_k$ are chosen to be primitive integer vectors. That is, $k$ is minimal among any possible such expressions.
  A cone $\si={\Cone}\,(P_1,\dots, P_k)$ is {\it simplicial} if  $\dim\,\si=k$ and $\si$ is {\it regular} if $P_1,\dots, P_k$ are primitive integer vectors which can be  extended to a basis of the lattice $\mathbb Z^n\subset \mathbb R^n$.

Recall that $f$ is {\em convenient} if $\Ga(f)$ touches with every coordinate axis. We say   $f$ is {\em pseudo-convenient} if $f$ is written as $f(\mathbf z)=\mathbf z^{\nu_0} f'(\mathbf z)$
where $f'$ is convenient and $\nu_0$ is a positive integer vector.
  \subsection{Toric modification}
  A regular simplicial cone subdivision $\Si^*$ of the space of positive weight vectors $N^+_{\mathbb R}=\mathbb R_+^n$ is  {\it admissible with the dual Newton diagram $\Ga^*(f)$} if $\Si^*$ is a subdivision of $\Ga^*(f)$. For such a regular simplicial cone subdivision, we associate a modification 
  $\hat\pi:X\to \mathbb C^n$ as follows:
  let $\mathcal S$ be the set of $n$-dimensional cones in $\Si^*$.
  For each $\si={\Cone} (P_1,\dots, P_n)\in \mathcal S$, we identify $\si$ with the unimodular matrix:
  \[
  \si=\left(\begin{matrix}
  p_{11},&\dots&p_{1n}\\
  \vdots&\vdots&\vdots\\
  p_{n1}&\dots&p_{nn}\end{matrix}
  \right)
  \]
  with $P_j={}^t(p_{1j},\dots, p_{nj})$. To each $\si\in \mathcal S$, we associate an affine coordinate chart  $(\mathbb C_\si^n,\mathbf u_\si)$ with  $\mathbf u_\si=(u_{\si 1},\dots, u_{\si n})$. The modification $\hat \pi$  is defined as follows. For each $\si\in \mathcal S$,
  we associate a birational mapping $ \hat\pi_\si:\mathbb C_\si^n\to \mathbb C^n$ by $z_i=u_{\si 1}^{p_{i1}}\dots u_{\si n}^{p_{in}}$
  for $i=1,\dots,n$ and $X$ is   the complex manifold obtained by gluing $\mathbb C_\si^n$ and $\mathbb C_\tau^n$ by $ \hat\pi_{\tau}\inv\circ\hat\pi_\si:\mathbb C_\si^n\to \mathbb C_\tau^n$ where it is well-defined.  This defines the modification $ \hat\pi:X\to \mathbb C^n$ which is proper and  the restriction  $\hat\pi$ to the torus ${\mathbb C_\si^{*n}}\subset \mathbb C_\si^n$ is an isomorphism onto the torus $\mathbb C^{*n}=(\mathbb C^*)^n$ in the base space. 
  If $\si={\Cone} (P_1,\dots, P_n)$ and $\tau={\Cone}(Q_1,\dots, Q_n)$ have a same vertex  $Q_1=P_1$,
the hyperplane $u_{\si 1}=0$ glues  canonically  with the hyperplane $\{u_{\tau 1}=0\}$. Thus 
 any vertex $P$ of $\Si^*$, gluing the hyperplanes on  every such toric coordinates with $P_1=P$,   defines a divisor in $X$,  and we denote this divisor by $\hat E(P)$.
If $P$ is strictly positive, $\hat E(P)$ is a compact divisor and $\hat\pi(\hat E(P))=\{O\}$. Recall that  a vertex of $\Si^*$ is a primitive integer generator of a 1-dimensional cone of $\Si^*$.
Note that if $f$ is pseudo-convenient, there exists a regular simplicial subdivision $\Si^*$ whose 
vertices are   strictly positive except for the canonical weight vectors
$e_i={}^t(0,\dots,\overset{\overset i{\smile}}1,\dots,0), i=1,\dots, n$.
For simplicity we assume hereafter that $f(\mathbf z)$ has a convenient or pseudo-convenient Newton boundary and the vertices of $\Si^*$ are strictly positive except the canonical weight vectors
$e_i, i=1,\dots, n$. Recall $\hat E(e_i)$ is bijectively mapped onto $\{z_i=0\}$.
Let $\mathcal V^+$ be the set of strictly positive vertices of $\Si^*$. Then the exceptional divisors of $\hat \pi:X\to\mathbb C^n$
corresponds bijectively to the vertices of $\mathcal V^+$.
The pull-back   ${\hat\pi}^*f$ of $f$ is expressed in the toric chart $\mathbb C_\si^n$ with $\si={\Cone}(P_1,\dots, P_n)$ as  follows:
\[
 {\hat\pi}^*f(\mathbf u_\si)=(\prod_{i=1}^n u_{\si,i}^{d(P_i)})\widetilde f(\mathbf u_\si)
\]
and $\widetilde f(\mathbf u_\si)$ is the defining function of the strict transform $\widetilde V$ of $V$.  The intersection $E(P):=\widetilde V\cap\hat E(P)$ is defined in $\hat E(P)$ by $\widetilde f(0,u_{\si 2},\dots,u_{\si n})=0$.
$E(P)$ is an exceptional divisor of the restriction $\pi:=\hat\pi|_{\widetilde V}:\widetilde V\to V$.
We recall that $\pi\inv(O)$ is the union of $E(P)$ such that  $P\in \mathcal V^+$ and  $\De(P)\ge 1$.
Two exceptional divisors $\hat E(P)$ and $\hat E(Q)$ intersect if and only if there is a $\si\in \mathcal S$ such that $\si={\Cone} (P,Q,P_3,\dots, P_n)$. However for $E(P)\cap E(Q)$ to be  non-empty,   besides the existence of such a $\si$, it is also 
necessary that  $\dim\, \De(P)\cap \De(Q)\ge 1$.
See Proposition (1.3.2), in Chapter II (\cite{Okabook}).

 If $f$ is Newton non-degenerate, any admissible toric modification $ {\hat\pi}:X\to \mathbb C^n$ gives a good resolution of $f$
 and by (AC) the zeta function  is written as 
\begin{eqnarray}\label{AC2}
\zeta(t)&=&\prod_{P\in \mathcal V^+} (1-t^{d(P)})^{-\chi(\hat E(P)'')}\\
\hat E(P)''&=&\hat E(P)\setminus \left(\widetilde V\cup_{Q\in \mathcal V^+,Q\ne P}\hat E(Q)\right)
\end{eqnarray}
where $\widetilde V$ is the strict transform of $V=f\inv(0)$ on $X$.
Let $\hat E(P)^{*I}:=\hat E(P)\cap \mathbb C^{*I}$ and $E(P)^I=E(P)\cap \mathbb C^{*I}$.
Then we use  the toric decomposition
$ \hat E(P)=\cup_I\hat E(P)^{*I}$ and the equality
\[
\chi({ E(P)^{*I}})=\chi(Q)=(-1)^{|I|-1}|I|!\Vol_{|I|}C(\De(Q)\cap \mathbb R^I,O)/d(Q)
\]
for the computation  of  the Euler characteristic 
$\chi(\hat E(P)'')$
 where $Q\in N^{I}$ is chosen to be a primitive integer vector such that 
$\De(Q)=\De(P)\cap \mathbb R^I$ (See  \cite{Ko, Ok3}).
This formula says that the Euler characteristics of $E(P)^{*I}$   is zero if $\dim\,\De(P)\cap \mathbb R^I< |I|-1$.


\section{Almost non-degenerate functions}
Consider a function $f(\mathbf z)=\sum_{\nu} a_\nu {\mathbf z}^\nu$ which is expanded in the Taylor series and let $\Gamma(f)$ be the Newton boundary.
Let $\hat\pi:X\to \mathbb C^n$ be a toric modification with respect to $\Si^*$ which is a simplicial regular subdivision of 
the dual Newton diagram $\Ga^*(f)$.
Let $\mathcal M$ be the set of maximal dimensional faces of $\Ga(f)$ and let $\mathcal M_0$ be the subset of $\mathcal M$ so that 
for $\De\in \mathcal M$, $f_\De:\mathbb C^{*n}\to\mathbb C$ is degenerate if and only if $\De\in \mathcal M_0$.
For $P$ which corresponds to a maximal face $\De\in \mathcal M$, we denote the exceptional divisor by
$\hat E(P)$ which corresponds to $P$.
We say that $f$ is {\em an almost non-degenerate function} if it satisfies the following conditions.
\begin{enumerate}
 \item[(A1)] {\em For any face $\De$ of $\Ga(f)$ with either $\De\in \mathcal M\setminus \mathcal M_0$ or $\dim\, \De\le n-2$,
$f$ is Newton non-degenerate on $\De$. For  $\De\in \mathcal  M_0$, $f_{\De}:\mathbb C^{*n}\to \mathbb C$ has a finite number of 1-dimensional critical loci which are $\mathbb C^*$-orbits through the origin. }
\end{enumerate}
Here we recall that 
$f_\De(\mathbf z)$ is a weighted homogeneous polynomial with respect to the weight vector $P$ and there is an associated $\mathbf C^*$-action defined by $t\circ (z_1,\dots, z_n)=(t^{p_1}z_1,\dots, t^{p_n}z_n),\,t\in \mathbf C^*$ where $\De(P)=\De$ and $P={}^t(p_1,\dots,p_n)$. Critical points loci of $f_\De$ are stable under this action.

Let $\si={\Cone} (P_1,\dots, P_n)$ be a simplicial cone in $\Si^*$ such that $\De(P_1)=\De\in \mathcal M_0$.
Let $\mathbf u_\si=(u_{\si 1},\dots, u_{\si n})$ be the corresponding toric coordinate chart.
The strict transform $\widetilde V$ of $V(f)$ is  defined  by $\widetilde f(\mathbf u_\si)=0$ where
$\widetilde f$ is defined by the equality:
\[
{ {\hat\pi}}^*f=\left(\prod_{i=1}^n u_{\si,i}^{d(P_i)}\right)\, \widetilde f(\mathbf u_\si)=0
\]
 and $E(P_1)\subset \hat E_0$ is defined by 
 $\{\mathbf u_\si\,|\,u_{\si 1}=0,g_\De(u_{\si 2},\dots, u_{\si n})=0\}$
 where $g_{\De}(u_{\si, 2},\dots, u_{\si n}):=\widetilde f(0,u_{\si 2},\dots,u_{\si n})$.
 The assumption (A1) implies $E(P_1)$ has 
 a finite number of  isolated singular points. Let  $S(\De)$ be the set of the singular points of $E(P_1)$.
 Take any $q\in S(\De)$ and assume $q=(0,\be_2,\dots, \be_n)$ in $\mathbb C_\si^n$.
An {\em admissible coordinate  chart at $q$}  is  an analytic coordinate chart $(U_q,\mathbf w)$, $\mathbf w=(w_1,\dots,w_n)$ centered at $q$ where $U_q$ is an open neighborhood of $q$ and $(w_2,\dots, w_n)$ is an analytic coordinate change of $(u_{\si 2},\dots, u_{\si n})$, but  we  do not change $u_{\si 1}$  and we always  assume $w_1=u_{\si 1}$.   (In many cases,  we can take $w_i=u_{\si,i}-\be_i,\,i=2,\dots,n$.)  As $w_1=u_{\si 1}$,   $w_1 =0$ is the defining function of $\hat E(P_1)$.
 As a second condition, we assume 
 \begin{enumerate}
 \item[(A2)]
{\em For any $\De\in \mathcal M_0$ and $q\in S(\De)$,  there exists an admissible coordinate $(U_q,\mathbf w)$ centered at $q$ such that  ${\hat \pi}^*f(\mathbf w)$ is Newton non-degenerate  and pseudo-convenient with respect to this coordinates $(U_q,\mathbf w)$.}
\end{enumerate}
\begin{Remark}{\rm
The conditions (A1),  (A2) can be generalized for   maximal  faces of $f^I$. The definition is similar.}
\end{Remark}
\subsection{First modification of $f$}
We assume that $f(z)$ is an almost non-degenerate function and we take an admissible toric modification as in the previous section.
We consider the tubular Milnor fibration
\[
(\star)\quad f:U(\eps,\de)^*\to D_\de^*,\, U(\eps,\de)=\{\mathbf z\,|\, 0<|f(z)|\le \de,\,\|\mathbf z\|\le \eps\},\,\de\ll \eps.
\]
This fibration  
is isomorphically lifted on $X$ so that 
$\hat f: \hat U(\eps,\de)^*\to D_\de^*$ is equivalent to the Milnor fibration  $(\star)$.
Here 
$ \hat U(\eps,\de)^*:={\hat\pi}\inv(U(\eps,\de))$ and $\hat f=f\circ\hat \pi$.
This fibration can be decomposed as the union of the fibrations along  $\hat E(P)$ for $P\in \mathcal V^+$ and local Milnor fibrations of $\hat f$ at  $q\in S(\De),\,\De\in \mathcal M_0$.
\begin{figure}[htb]  
\setlength{\unitlength}{1bp}
\begin{picture}(600,300)(-100,-30)
{\includegraphics[width=8cm, bb=0 0 595 842]{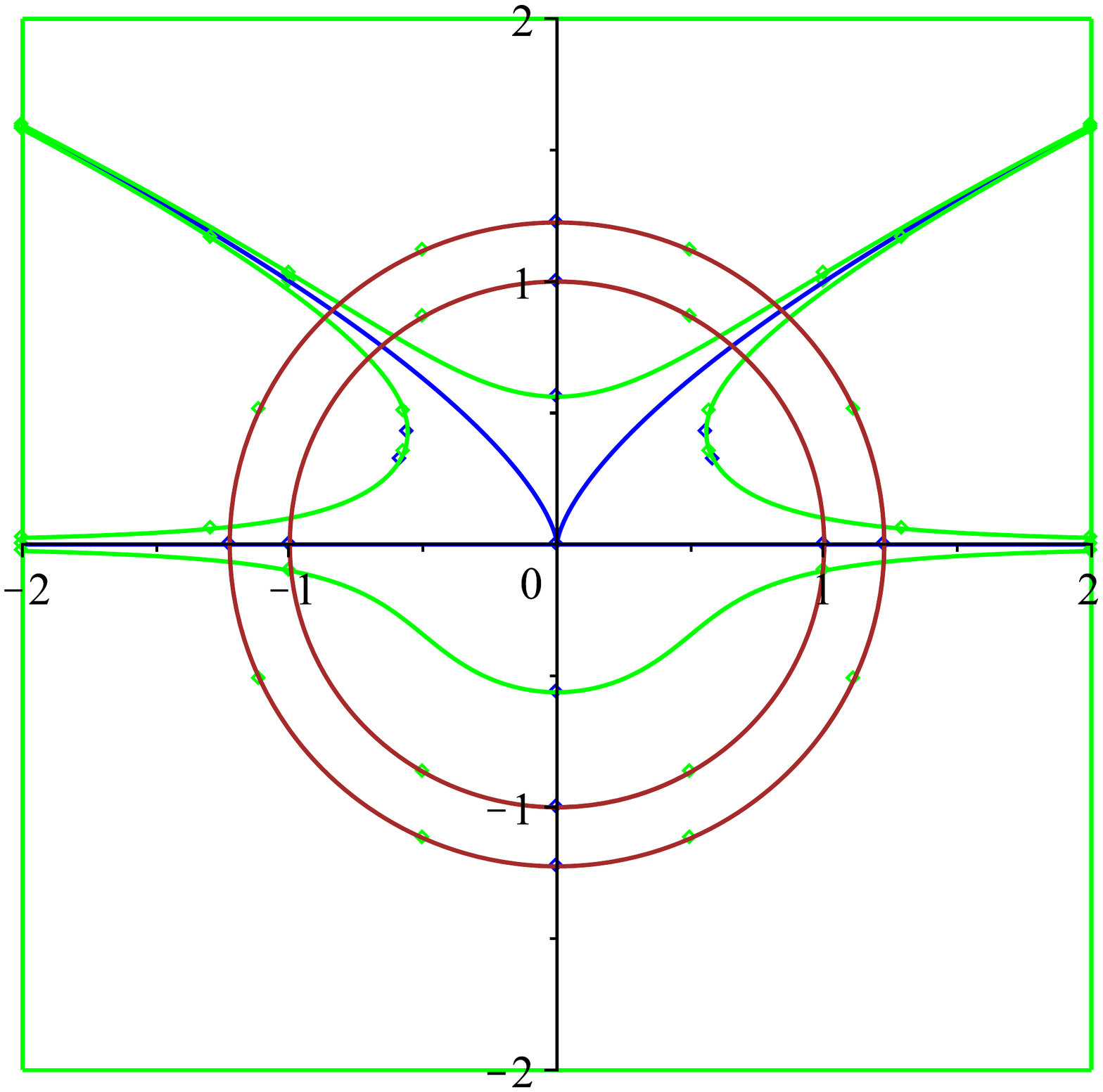}}
\put(-320,165){${\hat f}\inv(\pm\de)$: green line}
\put(-320,135){${\hat f}\inv(0)$: blue line}
\end{picture}
\vspace{-2cm}
\caption{Decomposition of Milnor Fibering}\label{MD}
\end{figure}
In Figure 1, brown circles are the spheres of radius $\eps'$ and $\eps''$ where $\eps''<\eps'$ and $\eps'-\eps''$ is sufficiently small.
At each $q\in S(\De)$, 
we take a small ball $B_{\eps'}(q)$ and we consider the local Milnor fibration of the function $\hat f(\mathbf w)={\hat \pi}^*f(\mathbf w)$   at $q$:
\[
\hat f: U_q(\eps',\de)^*\to D_{\de}^*,\]
where \[
U_q(\eps',\de)^*=\{\mathbf w\in U_q\,|\, 0<|\hat f(\mathbf w)|\le \de,\,\|\mathbf w\|\le \eps'\},\,\de\ll \min\{\eps',\eps\}.
\]
We  assume  $\de$ is small enough so that we can use the same  $\de$ in $(\star)$   for the local Milnor fibrations at $q$. This means that for any $\eta\ne0,\,|\eta|\le \de$, the level hypersurface $\hat f\inv(\eta)$ intersects transversely with the sphere $S^{2n-1}_{\eps'}(q)$.
Here we assume that  $\mathbf w$ is an admissible coordinate at $q$. Let us consider the decomposition of of the total space of the lifted  Milnor fibration $\hat U(\eps,\de)^*$,
\[\begin{split}
\hat U(\eps,\de)^*&=\hat U(\eps,\de)'\cup
\left(\cup_{q\in \De,\De\in \mathcal  M_0}U_q(\eps',\de)\right)\\
&\text{where}\,\, \hat U(\eps,\de)':=U(\eps,\de)^*\setminus 
\cup
\left(\cup_{q\in \De,\De\in \mathcal  M_0}U_q( \eps'',\de)\right)
\end{split}
\]
and we consider the corresponding decomposition of the Milnor fibration.  We take $\eps''$  a bit smaller than $\eps'$.
Assume   $\{q_1,\dots, q_m\}=\{q\in S(\De)\,|\, \De\in \mathcal  M_0\}$ 
and consider 
the sequence of subspaces
\[\begin{split}
\hat U_0(\eps,\de)'&\subset \hat U_1(\eps,\de)'\subset\cdots \subset \hat U_m(\eps,\de)=\hat U(\eps,\de)\\
&\text{where}\,\,\quad \hat U_i(\eps,\de):=\hat U(\eps,\de)'\bigcup_{j\le i}  U_{q_j}(\eps'\de)).
\end{split}
\]

Let $\zeta'(t)$ be the zeta function of the restriction $\hat f: \hat U(\eps,\de)'\to D_\de^*$
and let $\zeta_q(t)$ be the zeta function of the local Milnor fibration of  $\hat f$ at $q$.
We denote   the set of primitive weight  vectors  which correspond to $\mathcal M$ and $\mathcal M_0$
by $\mathcal P$ and $\mathcal P_0$ respectively.
Then we have 
\begin{Lemma} \label{key}
Assume that $f(\mathbf z)$ is an almost non-degenerate function as above. Then
the  zeta function of $f$ is given as the product
\begin{eqnarray}\label{Lemma}
\zeta(t)=\zeta'(t)\prod_{q\in S(\De)}\zeta_q(t)
\end{eqnarray}
where each product factor $\zeta_{q}(t)$ can be computed  by  
Varchenko's formula (V) and 
$\zeta'(t)$ is given as
\begin{eqnarray}\label{zeta-correction}
\quad\zeta'(t)&=&\prod_{I\subsetneq\{1,\dots,n\}}\zeta_I(t)\prod_{P\in \mathcal P\setminus \mathcal P_0}(1-t^{d(P)})^{-\chi(P)}\\
&&\times \prod_{P\in \mathcal P_0}(1-t^{d(P)})^{-\chi(P)+(-1)^{n-1}\sum_{q\in S(\De(P))}\mu_q}.
\notag
\end{eqnarray}
\end{Lemma}
\begin{proof} The  factor $\zeta_I(t)$ and $(1-t^d(P))^{-\chi(P)}$ in the first line of (\ref{zeta-correction}) are the same as in (V).
The   equality  (\ref{zeta-correction}) will be proved  \S \ref{zeta'}
where 
$\mu_q$ is the Milnor number of the hypersurface  $E(P)$ in $\hat E(P)$ at $q$. The proof of the assertion (\ref{Lemma}) is essentially the same as  the proof of Theorem (5.2), Chapter I, \cite{Okabook}.
We give an outline of a proof of  the assertion by an  inductive argument on $i$, showing
the zeta function $\zeta^{(i+1)}(t)$ of 
$\hat f:\hat U_{i+1}(\eps,\de)\to D_\de^*$ is given as 
\begin{eqnarray}\label{induction}
\zeta^{(i+1)}(t)=\zeta^{(i)}(t) \zeta_{q_{i+1}}(t),\,\,0\le i\le m.
\end{eqnarray}
For the proof, we use the following Sublemma \ref{Product} and Proposition \ref{Mayer}.
\begin{Sublemma}[Lemma (5.3), \cite{Okabook}]\label{Product}
Let $U\subset  {\hat\pi}\inv(U(\eps,\de)) $  and suppose that there is a manifold $M$ and a submersion $p:U\to M$
so that $p\times \hat f:U\to M\times D_\de^*$ is a locally trivial fibration. Its restriction to $p\inv(m),\,\hat f: p\inv(m)\to
D_\de^*$, with $m\in M$ is also a fibration.
Let $\zeta(t)$ and $\zeta_M^{\perp}(t)$ be the respective zeta functions of the fibrations $ {\hat f}:U\to D_\de^*$ and $\hat f: p\inv(m)\to D_\de^*$. Then we have the equality: $\zeta(t)=(\zeta_M^{\perp}(t))^{\chi(M)  }$.
\end{Sublemma}
The assertion is trivial when $p:U\to M$ is  a trivial fibration. Then we apply Mayer-Vietoris argument.
$\zeta_M^{\perp}(t)$ is called {\it the normal zeta function along $M$}.
\begin{Proposition}[Proposition (2.8), \cite{Okabook}]\label{Mayer}
Let $U=U_1\cup U_2$ be an open covering of the fibration $p:U\to D_\de^*$ where the restriction of $p$ to $U_1,U_2$ and $U_{12}:=U_1\cap U_2$ is also fibration. Consider four fibrations.
Let $F, F_1,F_2,F_{12}$ be the respective fibers and let $\zeta, \zeta_1(t),\zeta_2(t), \zeta_{12}$ be their zeta functions. Then 
\[
\chi(F)=\chi(F_1)+\chi(F_2)-\chi(F_{12}),\quad
\zeta(t)=\zeta_1(t)\zeta_2(t)\zeta_{12}(t)\inv.
\]
\end{Proposition}
The assertion follows easily from the Mayer-Vietoris argument.

We apply  Proposition \ref{Mayer} to the union
$\hat U_{i+1}(\eps,\de)=\hat U_i(\eps,\de)\cup U_{q_{i+1}}(\eps',\de)$.
Let
$W_{i+1}=U_{q_{j+1}}(\eps',\de)\setminus U_{q_{i+1}}( \eps'',\de)=\hat U_i(\eps,\de)\cap U_{q_{i+1}}(\eps',\de)$.  
The proof of the inductive assertion (\ref{induction}) follows from the next assertion.
\begin{Assertion} \label{Induction}
The contribution to the zeta function from $W_{i+1}=\hat U_i(\eps,\de)\cap U_{q_{i+1}}(\eps',\de)$ is trivial.
\end{Assertion}
Assuming Assertion \ref{Induction}, Lemma \ref{key} follows by the inductive argument. 
\end{proof}
\begin{proof}[Proof of Assertion \ref{Induction}] We take $\de\ll \eps,\eps'$.
Assume that $q_{i+1}\in \hat E(P),\,P\in \mathcal P_0$ and put
\begin{eqnarray*}
\hat E(P)_{i+1}'&=&\hat E(P)\cap (B_{\eps'}(q_{i+1})\setminus B_{ \eps''}(q_{i+1})),\quad\\
 E(P)_{i+1}'&=&E(P)\cap (B_{\eps'}(q_{i+1})\setminus B_{ \eps''}(q_{i+1}))
\end{eqnarray*}
where $B_r(q_{i+1})$ is the ball of radius $r$ with the center $q_{i+1}$.
$\hat E(P)_{i+1}'$ and $ E(P)_{i+1}'$ are non-singular and homotopically equivalent to $S^{2n-3}$
and to the link $K_{\eps'}:=g_P\inv(0)\cap S_{\eps'}^{2n-3}$ respectively. 
Note that  in the toric coordinate chart $\si={\Cone}(P_1,\dots, P_n)$ with $P=P_1$, $\hat f$ is written as $u_{\si 1}^{d(P)} \widetilde f(\mathbf u_\si)$ and $g_P(\mathbf u_\si')=\widetilde f(0,\mathbf u_\si')$ 
and 
$\mathbf u_\si'=(u_{\si 2},\dots, u_{\si n})$ and $g_P=0$ defines the hypersurface $E(P)$. Note that   $u_{\si 1},g_P(\mathbf u_\si')$ can be considered as a part of an analytic coordinate chart
at any $q_{i+1}'\in E(P)_{i+1}' $. 
Take a small tubular neighborhood $U_\ga$ in $X$
of $E(P)_{i+1}'$  of radius $\ga$ with some distance function from $E(P_1)$. $U_\ga$ is the set of points whose distance to $E(P)_{i+1}'$ is less than $\ga$. Consider  the union $W_{i+1}=W_{i+1}'\cup U_\ga'$
where $W_{i+1}'=\overline{W_{i+1}}\setminus U_{\ga/2}$
and $U_\ga'=U_\ga\cap \hat U(\eps,\de)$. We can identify $W_{i+1}'$
as a tubular neighborhood of $\hat E(P)_{i+1}'$ in $X$
over $\hat E(P)'\setminus U_\ga'$ 
and choose compatible projections  $p_1: W_{i+1}'\to \hat E(P)_{i+1}'$ and $p_2: U_\ga\to E(P)_{i+1}'$.
Here compatible means $p_2\circ p_1=p_1$ wherever   both sides are defined. (When we take a point $q_{i+1}'\in E(P)_{i+1}'$ and a small neighborhood of $q_{i+1}'$ so that $(u_{\si 1},g_P,u_{\si 3}',\dots, u_{\si n}')$ are coordinates, we can assume that
$p_1:(u_{\si 1},g_P,u_{\si 3}',\dots, u_{\si n}')\mapsto (0,g_P,u_{\si 3}',\dots, u_{\si n}')$ and $p_2: (u_{\si 1},g_P,u_{\si 3}',\dots, u_{\si n}')\mapsto (0,0,u_{\si 3}',\dots, u_{\si n}')$.)
First we consider $\hat f\times p_1: W_{i+1}'\to D_\de^*\times (\hat E(P)'\setminus U_{\ga/2})$.
The restriction of the Milnor fibration over a fiber of $p_1$ is a cyclic covering corresponding to $u_{\si 1}^{d(P)}$. Thus the normal zeta function along $\hat E(P)_{i+1}'$ is $(1-t^{d(P)})$. As $\hat E(P)_{i+1}'\setminus U_{\ga/2}$
is homotopic to the complement of the link
$K_{\eps'}$ in $S_{\eps'}^{2n-3}$,  the Euler characteristic is zero and  the zeta function of $ {\hat f}:W_{i+1}'\to D_\de^*$ is trivial by Sublemma \ref{Product}.
Next, we consider   $\hat f\times p_2: U_\ga'\to D_\de^*\times E(P)_{i+1}'$.
The normal zeta function along
$E(P)_{i+1}'$ 
is $1$, as it corresponds to the function of two variables $u_{\si 1}^{d(P_1)}g_P$ ( (3.7.3), Chapter I, \cite{Okabook}).
Thus combining the argument, the zeta functions of the restrictions of the Milnor fibration to  $W_{i+1}'$, to $U_\ga'$ 
and to $ W_{i+1}'\cap U_\ga' $ are all trivial. \end{proof}

\begin{Remark} {\rm The assertion (\ref{Lemma})  in 
Lemma \ref{key} is true without assuming the non-degeneracy of $\hat f$ at $q$. In that case, $\zeta_q(t)$ must  be computed without using  the formula (V). We need a practical resolution information of $\hat f$ at each singular point $q\in S(\De)$.
If $\hat f$ is non-degenerate at $q$ with respect to an admissible coordinate $\mathbf w$, we only need the information of $\Ga(\hat f,\mathbf w)$. Then we apply (V).}
\end{Remark}
\subsection{Zeta function $\zeta'(t)$ and the proof of the assertion   (\ref{zeta-correction}) in Lemma \ref{key}}\label{zeta'}
 Recall that $\zeta'(t)$ is the zeta function for the first toric modification, outside of the singular points $\{q\,|\, q\in S(\De),\,\De\in \mathcal M_0\}$. Applying   A'Campo's formula and Varchenko's description, we get a formula for the zeta function $\zeta'(t)$
 which  is given  as follows.

  For $\De\in \mathcal M\setminus \mathcal M_0$, the calculation is the same as  in the proof of Theorem (5.3), \cite{Okabook}.
 Assume that $\De\in \mathcal M_0$.
Let $P$ be the primitive weight vector which corresponds to $\De$ and take a toric chart $\mathbb C^n_\si,\,\si={\Cone}(P_1,\dots, P_n)$
with $P=P_1$ and let $\mathbf u_\si=(u_{\si 1},\dots, u_{\si n})$ the toric coordinate of this chart.
Then as in the previous section, the exceptional divisor $\hat E(P)$ is defined by 
$u_{\si 1}=0$ and $E(P):=\hat E(P)\cap \widetilde V$ is defined in $\hat E(P)$ by $g_P(u_{\si 2},\dots, u_{\si n})=0$ where
\[
\begin{split}
{ {\hat\pi}}^*f(\mathbf u_\si)&=\left(\prod_{i=1}^n u_{\si,i}^{d(P_i;f)}\right ) \widetilde f(\mathbf u_\si)\\
g_P(u_{\si 2},\dots, u_{\si n})&:=\widetilde f(0,u_{\si 2},\dots, u_{\si n}).
\end{split}
\]
Let $\mu_q$ be the Milnor number of $(g_P,q)$ as a germ of a hypersurface  at $q\in\hat E(P)$.
We take a small ball $B_q(\eps)$ centered at $q$ for  $q\in S(\De)$.
Let us consider a small perturbation family $f_s(\mathbf z),\,|s|\le 1$ of the coefficients of $f$ so that $f=f_0$, $\Ga(f_s)=\Ga(f)$
for any $s$ and $f_s$ is non-degenerate for $s\ne 0$. More precisely, we need only move a bit  the coefficients of $f_\De,\,\De\in \mathcal M_0$. The same toric modification $\hat \pi: X\to\mathbb C^n$ gives a good resolution of the family $f_s$ for any $s\ne 0$.
Let $\zeta^{(s)}(t)$ be the zeta function of $f_s,\,s\ne 0$. Let $\widetilde V_0=\widetilde V$ and $\widetilde V_s$ be the strict transform of $f\inv(0)$ and $f_s\inv(0)$ respectively.
Let $\hat E(P)_s''$ be the corresponding   factor of the exceptional divisor of $\hat E(P)$ which appears in the product expression of $\zeta^{(s)}(t)$ in the formula  (AC) of A'Campo and (V) of  Varchenko and let $E(P)_s$  and $E(P)_0$ ($=E(P)$ the intersections of $\hat E(P)$ and    the strict transforms $\widetilde V_s$ of $f_s=0$ and $\widetilde V_0$ of  $f_0=0$
respectively. Note that 
$\chi(E(P))=\chi(E(P)_s)+(-1)^{n-1}\sum_{q\in S(\De)}\mu_q$. Therefore
\[\begin{split}
\chi(\hat E(P)\setminus (\widetilde V_0\cup_{Q\ne P}\hat E(Q)))&=\chi(\hat E(P)\setminus (\widetilde V_s\cup_{Q\ne P}\hat E(Q))\\
&\qquad  -(-1)^{n-1}\sum_{q\in S(\De)}\mu_q.
\end{split}
\]
Thus the factor of zeta function  $\zeta'(t)$ coming from the divisor $\hat E(P)$ 
is changed from
\[
(1-t^{d(P)})^{-\chi(\hat E(P)_s'')}\quad\text{ to}\,\,
(1-t^{d(P)})^{-\chi(\hat E(P)_s'')+(-1)^{n-1}\sum_{q\in S(\De)}\mu_q}.
\]
 For $\De\in \mathcal M_0$, take  $P\in \mathcal P_0$ with $\De(P)=\De$ and we put
 \[\begin{split}
 \zeta_\De(t)&:=\prod_{q\in S(\De)}\zeta_q(t),\,\\
 \zeta^{er}_\De(t)&:=(1-t^{d(P)})^{(-1)^{n-1}\sum_{q\in S(\De)}\mu_q}\\
  \zeta^{er}(t)&:=\prod_{\De\in \mathcal  M_0}\zeta^{er}_\De(t)
 \end{split}
 \]
 Using  the above notations,
 we now obtain   the equality:
 \[
 \begin{split}
 \zeta'(t)&=\prod_{I\subsetneq\{1,\dots, n\}}\zeta_I(t)\prod_{P\in \mathcal P\setminus \mathcal P_0}(1-t^{d(P)})^{-\chi(P)}\\
 &\times \prod_{P\in \mathcal P_0}
  (1-t^{d(P)})^{ -\chi(\hat E(P)_s'')+(-1)^{n-1}\sum_{q\in S(\De)}\mu_q}\\
&= \zeta^{(s)}(t) \zeta^{er}(t).
 \end{split}
 \]
 This completes the proof of (\ref{zeta-correction}) in Lemma \ref{key}.
 \subsection{Second modifications}
For $\De\in \mathcal  M_0$ and $q\in S(\De)$, we choose  admissible coordinates $\mathbf w=(w_1,\dots,w_n)$ centered at $q$ and take an admissible regular simplicial subdivision $\Si_q^*$ of $\Ga^*(\hat f;\mathbf w)$. As $\hat f(\mathbf w)$ is pseudo-convenient, we assume that ($\sharp$) the vertices of $\Si_q^*$ are strictly positive except $e_1,\dots, e_n$. Then
we take the toric modification $\hat \omega_q: Y_q\to X$ with respect to $\Si_q^*$.  Taking the toric modification
at each $q\in S(\De)$, $\De\in \mathcal M_0$,
 let $\hat \omega: Y\to X$ is the union of the toric modification. 
 Here  $Y$ is the canonical gluing of the  union of $Y_q,\,q\in S(\De),\,\De\in \mathcal M_0$.
 The composition
\[
\Pi: Y\mapright{\hat\omega}X\mapright{ {\hat\pi}} \mathbb C^n
\]
 gives a good resolution of $f$. The exceptional divisors of $\Pi$ are all compact under the assumption ($\sharp$).
The zeta function $\zeta_q(t)$ is described by (AC) or (V).
Thus by Lemma \ref{key}, we have the following generalization of Varchenko's  formula:
\begin{Theorem} \label{main1}
\label{Main}The zeta function of $f$ is given by 
\[
\zeta(t)=\zeta^{(s)}(t)\zeta^{er}(t)\prod_{\De\in \mathcal M_0}\zeta_\De(t).
\]
where $\zeta^{(s)}(t)$ is the zeta function of the Newton non-degenerate function $f_s$ with $\Ga(f_s)=\Ga(f)$.
\end{Theorem}
\subsection{Examples}
{ Example 1}. Consider $f=(x-y)^2+y^3$.  $f$ has one face with the  weight vector $P={}^t(1,1)$. First toric modification $ {\hat\pi}:X\to \mathbb C^2$ is the ordinary blowing up.
Take the chart $(U_1,(u_{\si 1},u_{\si 2}))$ with $\hat \pi(u_{\si 1},u_{\si 2})=(u_{\si 1}u_{\si 2},u_{\si 1})$.
The pull back of $f$ is given as 
\[
 {\hat\pi}^*f(u_{\si 1},u_{\si 2})=u_{\si 1}^2((u_{\si 2}-1)^2+u_{\si 1}).
\]
Exceptional divisor is given by $u_{\si 1}=0$. 
Changing coordinate $w_1=u_{\si 1},w_2=u_{\si 1}-1$,
$ {\hat\pi}^*f=w_1^2(w_2^2+w_1)$ and we see that $\mu_q=1$.
The Newton boundary $\Ga( {\hat\pi}^*f,(w_1,w_2))$ has one face with  weight $Q={}^t(2,1)$ and $d(Q)=6$.
Thus by (V), $\zeta_q(t)=(1-t^6)(1-t^3)^{-1}$ where the factor $(1-t^3)^{-1}$ comes from the vertex $w_1^3$.
$\zeta(t)'=\zeta^{(s)}(t)\zeta^{er}(t)$, $\zeta^{(s)}(t)=1$ and $\zeta^{er}(t)=(1-t^2)^{-1}$ and $\zeta(t)=(1-t^6)(1-t^2)^{-1}(1-t^3)^{-1}$. 

\vspace{.3cm}
{ Example 2.} Consider $f=z^3+y^3+z^3-3xyz+z^4$. 
Note that $x^3+y^3+z^3-3 xyz=0$ consists of three planes
$x+y+z=0, x+\omega y+\omega^2 z=0, x+\omega^2 y+\omega z=0$
with $\omega=\exp(2\pi i/3)$. The dual Newton diagram has a single strictly positive vertex $P={}^t(1,1,1)$ and it is already regular simplicial cone subdivision. The corresponding toric modification is nothing but the ordinary blowing up.
Here an ordinary blowing up is a toric modification with one strictly positive vertex $P={}^t(1,1, 1)$. It has $3$ canonical toric charts.
After one blowing up, we work
in the toric coordinate chart $\mathbb C_\si^3,\, \mathbf u_\si=(u_{\si 1},u_{\si 2},u_{\si 3})$ where
$\si={\Cone}(P,e_1,e_2),\,P={}^t(1,1,1)$ and $x=u_{\si 1}u_{\si 2}, y=u_{\si 1}u_{\si 3},  z=u_{\si 1}$. 
We have
\[\begin{split}
\pi^*f&=u_{\si 1}^3\tilde f(\mathbf u_\si),
\,\\
\tilde f(\mathbf u_\si)&=(u_{\si 2}+u_{\si 3}+1)(u_{\si 2}+\omega u_{\si 3}+\omega^2)(u_{\si 2}+\omega^2 u_{\si 3}+\omega)+u_{\si 1}
\end{split}
\]
and 
$E(P)$ is defined  by 
$u_{\si 1}=0,\,\tilde f(0,u_{\si 2},u_{\si 3})=0$ which consists of three $\mathbb P^1$
and  three singular points are the intersection points of the lines which are  $q_0:=(0,1,1)$, $q_1=(0,\omega,\omega^2)$ and $q_2=(0,\omega^2,\omega)$.
For example, at $q_0$, taking the coordinate $w_1=u_{\si 1},w_2=u_{\si 2}-1, w_3=u_{\si 3}-1$, $ {\hat\pi}^*f$ is written as 
\[\begin{split}
 {\hat\pi}^*f&=u_{\si 1}^3(u_{\si 3}^3+u_{\si 2}^3-3u_{\si 3}u_{\si 2}+u_{\si 1}+1)\\
&=w_1^3 (w_3^3+w_2^3+3w_3^2-3w_3w_2+3 w_2^2+w_1)
\end{split}
\]
 By the symmetry of the equation, the singularities of $E(P)$ are isomorphic at any $q_i$.  They are $A_1$ singularity and $\mu_{q_i}=1$. 
Thus  $\zeta^{(s)}(t)=(1-t^3)^{-3}$, $\zeta^{er}(t)=(1-t^3)^3$ and   $\zeta'(t)=1$.
$\zeta_{q_i}(t)$ is given as 
$(1-t^4)^{-1}$. Thus   we get
\begin{Assertion}
$\zeta(t)=(1-t^4)^{-3}$ and 
$\mu=11$.
\end{Assertion}
The second assertion follows from $-\deg\, \zeta=1+\mu=12$.

The above calculation can be generalized for
$f_n=x^3+y^3+z^3-3xyz+z^n,\,n\ge 4$.
\begin{Assertion}
The zeta function of $f_n$ is given by $(1-t^n)^{-3}$ and $\mu(f_n)=3n-1$.
\end{Assertion}
In fact, after one blowing up, $\widetilde V$ has three singularities which are defined (up to isomorphism) by 

\[\begin{split}
 {\hat\pi}^*f &=u_{\si 1}^3(u_{\si 2}^3+u_{\si 3}^3-3u_{\si 2}u_{\si 3}+u_{\si 1}^{n-3}+1)\\
&=w_1^3((w_3+1)^3+(w_2+1)^3 -3(w_3+1)(w_2+1)+1+w_1^{n-3})\\
&=w_1^3( 3w_3^2+3 w_2^2-3w_3w_2+w_1^{n-3}).
\end{split}
\]
\vspace{.3cm}
{ Example 3}. Let $f_d(x, y, z)$ be an irreducible convenient  homogeneous polynomial which defines a projective curve of degree $d$ with 
$k\le \frac{(n-1)(n-2)}2$ nodes. See \cite{OkaFermat} for an example of  a maximal nodal curve.
We consider $f=f_d+x^{d+1}$.
As an affine polynomial, $f_d$ has a single maximal  dimensional face with weight vector $P={}^t(1,1,1)$ and it is Newton degenerate.
After one ordinary  blowing up with exceptional divisor $\hat E_0$, 
using coordinate $\mathbf u_\si,\,(x,y,z)=(u_{\si 1},u_{\si 1}u_{\si 2},u_{\si 1}u_{\si 3})$, $ {\hat\pi}^*f$ is written as 
\[
 {\hat\pi}^*f(\mathbf u_\si)=u_{\si 1}^d (\hat f_d(1,u_{\si 2},u_{\si 3})+u_{\si 1})
\]
and the strict transform $\widetilde V$ is defined by $\hat f_d(1,u_{\si 2},u_{\si 3})+u_{\si 1}=0$, $\hat E_0$ is defined by $u_{\si 1}=0$ and $\widetilde V\cap \hat E_0$ has $k$ nodal singularities at $\{q_1,\dots, q_k\}$.
At each point, $f(1,u_{\si 2},u_{\si 3})$ is written as 
\[\hat f_d(1,u_{\si 1},u_{\si 2})=Q(w_2,w_3)+R(w_2,w_3)
\]
where $\mathbf w$ is the admissible coordinate at $q_i$ and $Q(u_{\si 2},u_{\si 3})$ is a non-degenerate quadratic form,  and  we can assume $Q=w_2^2+w_3^2$. The last term $R(w_2,w_3)$ is a polynomial of degree greater than or equal 3.

We first consider the zeta function of $f_d$. Thus $ {\hat\pi}^*f_d(\mathbf w)$ is  equivalent to
$w_1^d(w_2^2+w_3^2+R(w_2,w_3)$ which is non-degenerate and its zeta function is trivial at $q_i$.
Thus by Theorem \ref{Main}, the zeta function of $f_d$ is given as $\zeta_{f_d}(t)=(1-t^d)^{-d^2+3d-3+k}.$
Recall $\zeta_{f_d}(t)$  is equal to the product $P_0(t)\inv P_1(t)P_2(t)\inv$
where
$P_j(t)$ is the characteristic polynomial of the monodromy action $h_{*j}: H_j(F;\mathbb Q)
\to H_j(F;\mathbb Q)$
and  $F=f_d\inv(1)$ is the Milnor fiber of $f_d$. By \cite{Fulton} (see also \cite{Deligne}), $\pi_1(\mathbb P^2\setminus \{f_d=0\})=\mathbb Z/d\mathbb Z$ and the canonical mapping $p: F\to \mathbb P^2\setminus \{f_d=0\})$ is $d$-cyclic covering.
This implies $F$ is simply connected and thus $P_1(t)=1$ and 
we get 
\[
\zeta_{f_d}(t)=(1-t^d)^{-d^2+3d-3+k},\quad P_2(t)={(1-t^d)^{d^2-3d+3-k}}{(1-t)^{-1}}.
\]
The polynomial  $f_d$ gives an example of a function  which has one dimensional singularities but still the Milnor fiber is  $1$-connected.

To compute the zeta function of $f$, we need to take one more blowing up at each singular point
 $q_i=(\al_i,\be_i),\,i=1,\dots,k$. We can choose  admissible local coordinates $\mathbf w_i=(w_{i1},w_{i2},w_{i3})$ with $w_{i1}=u_{\si 1}$ and $(w_{i2},w_{i3})$ is a linear change of $(u_{\si 2}-\al_i,u_{\si 3}-\be_i)$ so that 
\[
\hat \pi^*f=w_{i1}^d\left(w_{i2}^2+w_{i3}^2+w_{i1}+\text{(higher terms)}
\right)
\]
where (higher terms) contains only variables $w_{i2}$ and $w_{i3}$.
Thus the local zeta function is $\zeta_{q_i}(t)=(1-t^{d+1})^{-1}$ and we get
\[
\zeta(t)=(1-t^d)^{-d^2+3d-3+k}(1-t^{d+1})^{-k},\,\mu(f)=(d-1)^3+k.
\]
The last equality is generalized in Theorem \ref{MilnorFormula} in \S 5.
\section{Application:  A Zariski pair of links}
It is well-known that there exists a pair of projective curves $\{C,C'\}$ of degree 6 (so called a Zariski pair) with 6 cusps   whose complements have different topologies (Zariski \cite{Zariski}).
The first curve $C$ is sextic of torus type. A typical one is defined as follows:
\[\begin{split}
C: =&\{(x,y,z)\in \mathbb P^2\,|\, f_6(x,y,z)=0\}\\
f_6=&f_2^3+f_3^2,\,
f_2=x^2+y^2+z^2,\,f_3=x^3+y^3+z^3.
\end{split}
\]
Another curve $C'$ is a sextic with 6 cusps such that  there exists no conic which passes through these 6 points. We use the  sextic which is given in \cite{Okasymmetric}. For our purpose, we took the change of coordinates $(x,y)\mapsto (x+1/2,y+2)$.
This change of coordinates is simply to put the singular locus off the coordinate planes.
\begin{eqnarray}\label{NTS}
C': & &g_6(x,y,z)=0\notag\\
&&\quad g_6=-{\frac {215\,{z}^{6}}{64}}+{\frac {51\,x{z}^{5}}{16}}+{\frac {63\,{x}
^{2}{z}^{4}}{16}}-{\frac {3\,{x}^{3}{z}^{3}}{2}}-{\frac {9\,{x}^{4}{z}
^{2}}{4}}\\
&&+3\,{x}^{5}z+{x}^{6}-{\frac {41\,y{z}^{5}}{4}}
+8\,xy{z}^{4}+
10\,{x}^{2}y{z}^{3}-4\,{x}^{4}yz-{\frac {571\,{y}^{2}{z}^{4}}{48}}\notag\\
&&+{
\frac {22\,x{y}^{2}{z}^{3}}{3}}
+{\frac {47\,{x}^{2}{y}^{2}{z}^{2}}{6}}
-{x}^{4}{y}^{2}
-{\frac {190\,{y}^{3}{z}^{3}}{27}}+{\frac {8\,x{y}^{3}{
z}^{2}}{3}}\notag\\
&&+{\frac {8\,{x}^{2}{y}^{3}z}{3}} 
-{\frac {85\,{y}^{4}{z}^{2}
}{36}}+{\frac {x{y}^{4}z}{3}}+{\frac {{x}^{2}{y}^{4}}{3}}-{\frac {4\,{
y}^{5}z}{9}}-{\frac {{y}^{6}}{27}}\notag
\end{eqnarray}
To make the singularity of $f_6\inv(0)$ and $g_6\inv(0)$ to be isolated at the origin, we put
\[\begin{split}
f=f_6+z^7,\quad g=g_6+z^7
\end{split}
\]
and we consider the corresponding hypersurface   $ f\inv(0)$ and $g\inv(0)$ at the origin. 
We will show that they have the same Newton boundary,  the same zeta function (thus the same Milnor number too),
the same dual  resolution graph and thus  their links $M_f$ and $M_g$
are  diffeomorphic
where $M_f=V(f)\cap S_{\eps}^{2n-1}$ and $M_g=V(g)\cap S_\eps^{2n-1}$ and $\eps$
is small enough.  Their tangent cones gives a Zariski pair in $\mathbb P^2$. We call such a pair of links of  hypersurfaces $V(f), V(g)$  {\em a Zariski pair of links.}
\subsubsection{Torus type sextic as the tangent cone} We consider first $V=f\inv(0)$.
Consider the polynomials:
\[\begin{split}
&
f_2=x^2+y^2+z^2,\,f_3=x^3+y^3+z^3,\, \\
&f_6=f_2^3+f_3^2,\,\,
f=f_6+z^7.
\end{split}
\]
We first take   an ordinary blowing up $ {\hat\pi}: X\to \mathbb C^3$ and we take the chart
$(U,\mathbf u_\si)$, $\mathbf u_\si=(u_{\si 1},u_{\si 2},u_{\si 3})$ with $ {\hat\pi}(\mathbf u_\si)=(u_{\si 1}u_{\si 3},u_{\si 2}u_{\si 3},u_{\si 3})$. Denote the exceptional divisor  by $\hat E_0$ which is diffeomorphic to $\mathbb P^2$ and it 
 is defined by $u_{\si 3}=0$ in $U$. $E_0=\hat E_0\cap \widetilde V$ is the exceptional divisor of the restriction of $\hat \pi$ to $\widetilde V$,
 corresponding to the strict transform of $f_6=0$.
 More precisely we have:
 \[\begin{split}
 & {\hat\pi}^*f=u_{\si 3}^6 (\bar f_6+u_{\si 3}),\,
 \bar f_6={\bar f_2}^3+{\bar f_3}^2,\\
 &\bar f_2=u_{\si 1}^2+u_{\si 2}^2+1,\quad
 \bar f_3=u_{\si 1}^3+u_{\si 2}^3+1.
 \end{split}
 \]
 Let $\widetilde V$ be the strict transform of $V$.
 6 singular points, say $\rho_1,\dots, \rho_6$, of $E_0$ are located at the intersection of curves  $\bar f_2=\bar f_3=0$  in $E_0$ and they are $A_2$ singularities. On each intersection,  in $E_0$ the curves $\bar f_2=0$ and $\bar f_3=0$ are non-singular and intersect transversely.
  As $V(\bar f_2), V(\bar f_3), E_0=V(u_{\si 3})$  intersect transversely
 at $\rho_i$, we can take 
  $w_1=\bar f_2,\,w_2=\bar f_3,w_3=u_{\si 3}$
   as 
    analytic coordinates in a small neighborhood $U_i$ of  $\rho_i$ and  $(U_i\cap E_0,(w_1,w_2)$ is a local coordinate chart of $E_0$. Then 
    the pull-back is written as 
 \[
  {\hat\pi}^*f=w_3^6(w_1^3+w_2^2+w_3).
 \]
 The dual Newton diagram $\Ga^*(\hat \pi^*f(\mathbf w))$ is as  Figure \ref{DN} and  we take a regular subdivision $\Si_{\rho_i}^*$ as in Figure \ref{DN}.
 \begin{figure}[htb]  
\setlength{\unitlength}{1bp}
\begin{picture}(600,300)(-50,-30)
{\includegraphics[width=8cm, bb=0 0 555 555]{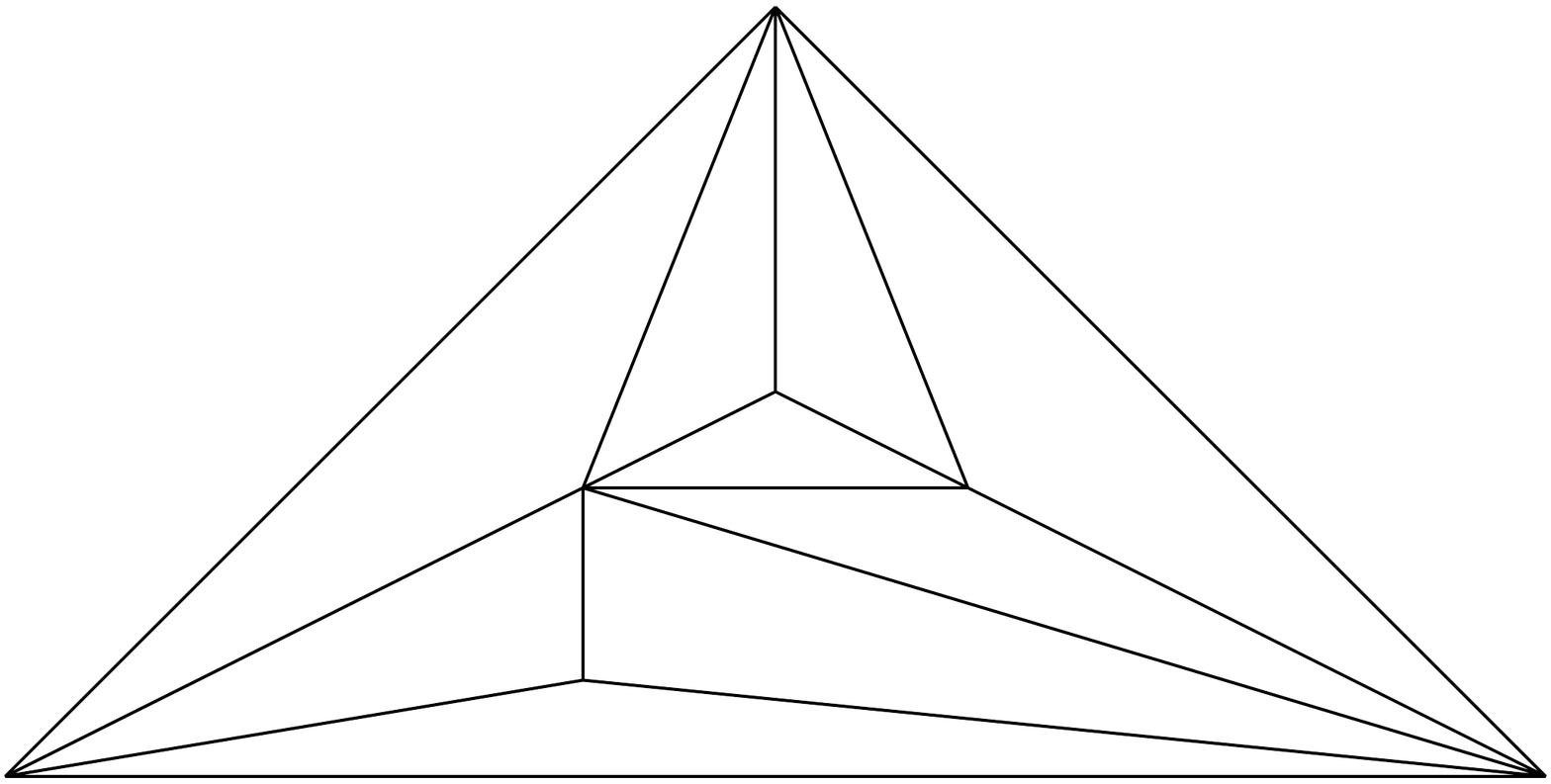}}
\put(-106,232){\circle*{7}}
\put (-98,230) {$e_3$}
\put(-226, 113){\circle*{7}}
\put(-242,110){$e_1$}
\put(15,113){\circle*{7}}
\put(22,110){$e_2$}
\put(-105, 172){\circle*{7}} 
\put(-100,170){$P$}
\put(-200,170){$\Si^*:$}
\put(-76, 158){\circle{7}}
\put(-70,155){$T$}
\put(-135, 157){\circle{7}}
\put(-148,153){$S$}
\put(-135,  127){\circle{7}}
\put(-130,125){$R$}
\put(-270,80){$P={}^t(2,3,6),\,T={}^t(1,2,3),\,S={}^t(1,1,2),\,R={}^t(1,1,1)$}
\end{picture}
\vspace{-4cm}
\caption{Dual Newton diagram}\label{DN}
\end{figure}
The bullets in Figure \ref{DN} are the vertices of  $\Ga^*( {\hat\pi}^*f)$ and by adding three vertices $S,T,R$, we get a  regular  simplicial subdivision $\Si_{\rho_i}^*$.
We take the  associated toric modification $\hat\omega_i: Y_i\to U_i$. It has four exceptional divisors $\hat E(P),\,\hat E(T),\, \hat E(S),\,\hat E(R)$.
The strict transform $\widetilde V_i$ of $\widetilde V$ is smooth.  The restriction $ {\hat\omega}_i: \widetilde V_i\to \widetilde V$ is a resolution of $\widetilde V$ at $\rho_i$ and it has three exceptional divisors
$E(P)=\hat E(P)\cap \widetilde V_i$, $E(T)=\hat E(T)\cap \widetilde V_i$ and $E(S)=\hat E(S)\cap \widetilde V_i$. On the other hand, 
$\hat E(R)\cap \widetilde V_i$ is empty as $\De(R)=\{(0,0,7)\}$. (See for example Proposition (3.7), page 131, \cite{Okabook}.)
 $E(e_1)$, $E(e_2)$  are strict transforms of  the conic $f_2=0$ and the cubic $f_3=0$.
 $E(e_3)$ is (the pull-back of) the exceptional divisor $E_0$. To distinguish divisors over other singular points $\rho_i, 1\le i\le 6$, we denote them by $E(P)_i,\,E(T)_i,\, E(S)_i$.
 We do  the same toric modification at $\rho_1,\dots, \rho_6$ and let $ {\hat\omega}: Y\to X$ be the union of these modifications
 and let $ {\hat\Pi}=\hat\pi\circ  {\hat\omega}: Y\to X\to \mathbb C^n$ the composition of the modifications.
 The exceptional divisors of $\hat \Pi$ is given as 
  \[D:= {\hat\Pi}\inv(0)\cap \widetilde V= E_0+\sum_{i=1}^6(E(P)_i+E(T)_i+E(S)_i).
  \]
   Note that
 $E(P)_i,E(T)_i, E(S)_i$ are isomorphic to $\mathbb P^1$. See for example,  Lemma (6.4), p.158, \cite{Okabook}. By abuse of notation, we denote $\hat\omega^*(E_0)$ by $E_0$. As for $E_0$, we assert:
 \begin{Assertion}
 $E_0$  has genus 4.
 \end{Assertion}
 \begin{proof}The assertion follows from the Euler characteristic calculation,  
 $\chi(E_0) = -18+6\cdot 2=-6$. 
 Here $-18$ is the Euler characteristic of the smooth sextic 
 and $12$ is the defect from 6 cusps.
 \end{proof}
 The link of $V$ is diffeomorphic to the boundary of the tubular neighborhood of the total exceptional divisor $D$.
 To study  geometry further, we consider the divisor defined by $ {\hat\Pi}^*f_2=0$ and compute the intersection numbers.
 We use the property $({\hat\Pi}^*f_2)\cdot C=0$ for any compact divisor (see for example Theorem 2.6 of  \cite{La}).
 We use three toric charts
 $\si={\Cone}(P,T,e_3)$, $\tau={\Cone}(P,S,e_3)$ and $\xi={\Cone}(S,e_1,e_3)$  with respective toric coordinates 
 $(u_{\si 1},u_{\si 2},u_{\si 3})$,
 $(u_{\tau 1},u_{\tau 2},u_{\tau 3})$ and $(u_{\xi 1},u_{\xi 2},u_{\xi 3})$.
 By an easy computation, we get
 \[\begin{split}
  {\hat\Pi}^*f_2&=u_{\si 3}^2(u_{\si 1}^2+u_{\si 2}^2+1)\\
 &=u_{\si 1}^{14} u_{\si 2}^7u_{\si 3}^2=u_{\tau 1}^{14}u_{\tau 2}^5 u_{\tau 3}^2=u_{\xi 1}^5u_{\xi 2}u_{\xi 3}^2
  \end{split}
 \]
 Thus we have 
 \[( {\hat\Pi}^*f_2)=2E_0+\widetilde V_2+\sum_{i=1}^6 (14 E(P)_i+7E(T)_i+5E(S)_i).
 \]
 Here $\widetilde V_{2}$ is the intersection of the strict transforms of $\widetilde V(f_2)\cap \widetilde V(f)$. 
 Note that $\widetilde V_2$ intersects  only with $E(S)_i,\,1\le i\le 6$. Let $I_f$ be the $19\times 19$ intersection matrix of the exceptional divisors.
 Thus we conclude
 \begin{eqnarray}\label{divisorMf}
 E(P)_i^2=-1,\,E(T)_i^2=-2,\,E(S)_i^2=-3,\, E_0^2=-42,\, \det\, I_f=6.
 \end{eqnarray}
  \vspace{.2cm}
  The total configuration graph is one vertex $E_0$ at the center and 6 branches, $\Ga_1,\dots, \Ga_6$, are is joined  to the center divisor $E_0$ through $E(P)_i$ for $i=1,\dots, 6$.
  \[
 \Ga_i:\quad  \overset{E(S)_i}\bullet  \rule[1mm]{1cm}{0.3mm}
 {\overset{E(P)_i}\bullet} \rule[1mm]{1cm}{0.3mm} \overset{E(T)_i}\bullet
  \]
  See Figure \ref{Graph} for the whole resolution graph.
  \subsubsection{Non Torus type sextic as the tangent cone}
  We consider $V(g)$ where 
  $g=g_6+z^7$ and $g_6$ is given in (\ref{NTS}). First we take a blowing up at the origin.
  Using the same coordinates $(u_{\si 1},u_{\si 2},u_{\si 3})$,
  \begin{eqnarray*}
  & &{\hat\pi}^*g=u_{\si 3}^6(\bar g_{6}(u_{\si 1},u_{\si 2})+u_{\si 3})\\
  &&\bar g_6(u_{\si 1},u_{\si 2})=1728\,{{  u_{\si 1}}}^{6}+5184\,{{  u_{\si 1}}}^{5}+ \left( -1728\,{{  u_{\si 2}}}^{2
}-6912\,{  u_{\si 2}}-3888 \right) {{  u_{\si 1}}}^{4}\\
&&-2592\,{{  u_{\si 1}}}^{3}+
 \left( 576\,{{  u_{\si 2}}}^{4}+4608\,{{  u_{\si 2}}}^{3}+13536\,{{  u_{\si 2}}}^{2}
+17280\,{  u_{\si 2}}+6804 \right) {{  u_{\si 1}}}^{2}\\
&&+ \left( 576\,{  u_{\si 2}}^{4}+4608\,{{  u_{\si 2}}}^{3}+12672\,{{  u_{\si 2}}}^{2}+13824\,{  u_{\si 2}}+5508
 \right) {  u_{\si 1}}-64\,{{  u_{\si 2}}}^{6}\\
 &&-768\,{{  u_{\si 2}}}^{5}-4080\,{{  
u_{\si 2}}}^{4}-12160\,{{  u_{\si 2}}}^{3}-20556\,{{  u_{\si 2}}}^{2}-17712\,{  u_{\si 2}}-
5805
  \end{eqnarray*}
  The graph of $\bar g_6=0$ is given in Figure \ref{bar-g6}.
  The singular points are 
  \[\begin{split}
  &\rho_1=(\frac12,-1,0),\, \rho_2=(\frac 12,-3,0), \,
  \rho_3=(-\frac 12,-1,0),\,\\
  &\rho_4=(-\frac 12,-3,0),\,\rho_5=(-\frac 12+\frac{\sqrt 6}3,-\frac 12,0),\,
  \rho_6=(-\frac 12-\frac{\sqrt 6}3,-\frac 12,0).
  \end{split}
  \]  
  and the tangent cones are  vertical for  $\{\rho_1,\rho_2\}$ and  horizontal  for   $\{\rho_3, \rho_4, \rho_5, \rho_6\}$.
  This implies the elementary choice of coordinates are admissible.
  The resolutions are similar for $\rho_1,\rho_2$ and also the resolutions for $\rho_3,
  \dots, \rho_6$ are similar. So we will see two resolutions at $\rho_1$ and $\rho_3$.
  
  \begin{figure}[htb]  
\setlength{\unitlength}{1bp}
\begin{picture}(600,300)(-50,-30)
{\includegraphics[width=8cm, bb=0 0 595 842]{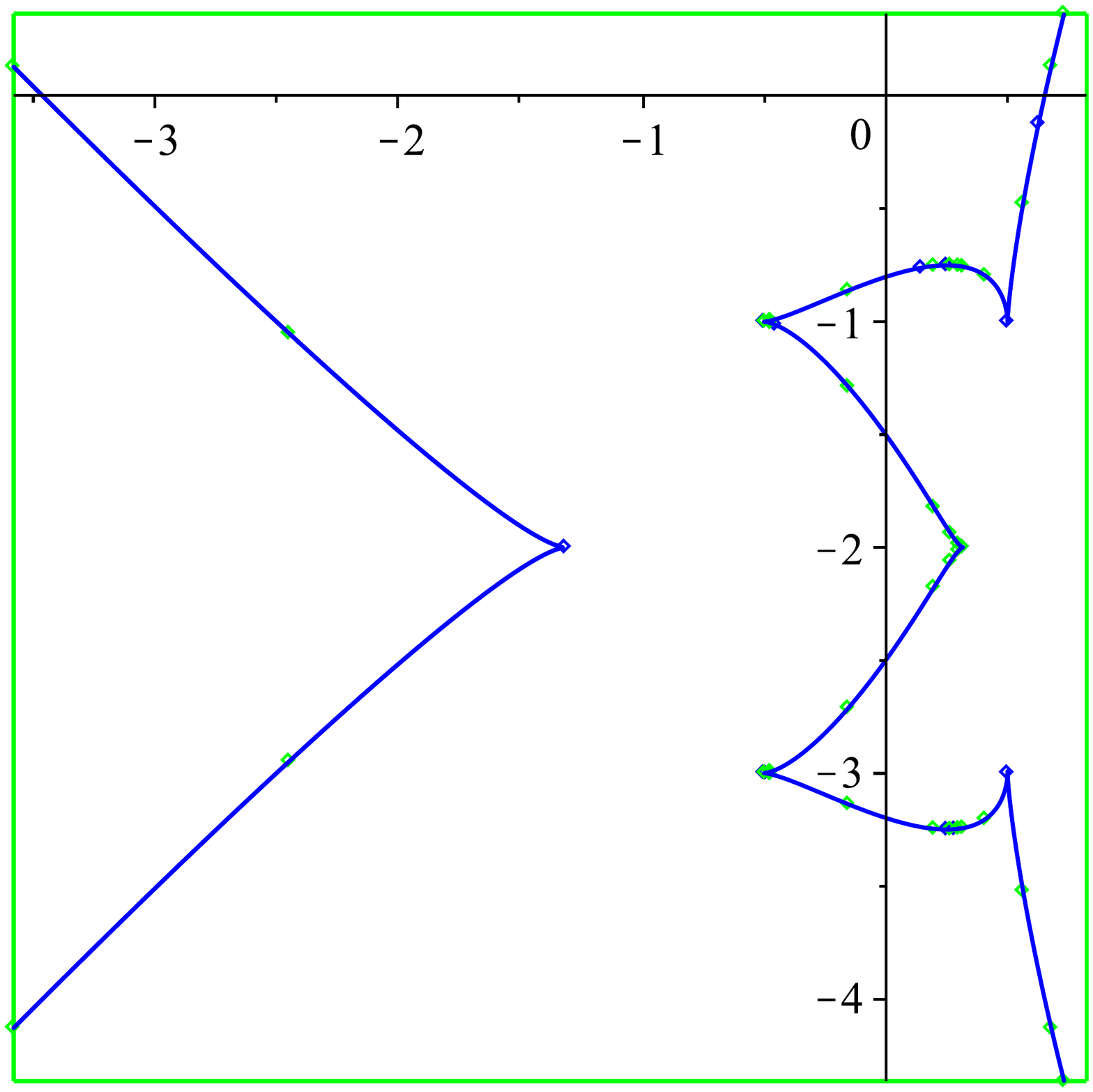}}
\put (-30,198) {$\rho_1$}
\put(-30, 120){$\rho_2$}
\put(-80,198){$\rho_3$}
\put(-80,120){$\rho_4$}
\put(-28,158){$\rho_5$}
\put(-104,158){$\rho_6$}
\end{picture}
\vspace{-2cm}
\caption{Graph of $\bar g_6=0$}\label{bar-g6}
\end{figure}
We use the pull back of $x$ for the calculation of the intersection numbers.
As none of $\rho_i$ is on $\{u_{\si 1}=0\}\subset \mathbb C_\si^3$ and 
$\hat\pi^*x=u_{\si 1}u_{\si 3}$,  
it is enough to compute the pullback of $u_{\si 3}$ to compute the divisor  $ ({\hat\pi}^*x)$.

\noindent
(1) First we consider the resolution at $\rho_1=(\frac 12,-1,0)$. Taking the coordinate $w_1=u_{\si 1}-1/2,\,w_2=u_{\si 2}+1,\,w_3=u_{\si 3}$,
 we have
\[\begin{split}
& {\hat\pi}^*g=w_3^6(\bar g_6(w_1,w_2)+w_3),\\
&\bar g_6(w_1,w_2)=-512 w_2^3+5184w_1^2+\text{(higher terms)},\\
& {\hat\pi}^*x=(w_1+1/2) w_3. 
\end{split}
\]
Thus the Newton boundary of $ {\hat\pi}^*g$ is the same as $\Ga^*(\hat\pi^*f(\mathbf w))$ in the previous section (i.e.,  $\Ga(w_3^6(-512 w_2^3+5184w_1^2))$)
and we can take the exact same regular simplicial subdivision $\Si_{\rho_i}^*$. Taking the toric modification $ {\hat\omega}_1:Y_1\to X$ with respect to $\Si_{\rho_i}^*$,
\[\begin{split}
 {\hat\omega}_1^*(w_3)&=u_{\si 1}^6u_{\si 2}^3u_{\si 3},\,\,\text{in} \,\,
\mathbb C_\si^3,\,\si={\Cone}(P,T,e_3)\\
&=u_{\tau 1}^6u_{\tau 2}^2u_{\tau 3}, \, \text{in}\,\mathbb C_\tau^3,\,\tau={\Cone}(P,S,e_3)\\
&=u_{\xi 1}^2u_{\xi 3},\,\,\text{in}\,\mathbb C_\xi^3,\,\xi={\Cone}(S,e_1,e_3).
\end{split}
\]
Thus $( {\hat\omega}_1^*w_3)=6E(P)_1+3E(T)_1+2E(S)_1+E_0$
and 
\begin{eqnarray}\label{rho1}
( {\hat\omega}_1^* {\hat\pi}^*x)=6E(P)_1+3E(T)_1+2E(S)_1+E_0+(u_{\si 1}=0) \,\,\text{in}\,\,Y_1.
\end{eqnarray}
The same equality for $\hat \omega_2: Y_2\to X.$

\noindent
(2) Consider the resolution at $\rho_3=(-\frac 12,-1,0)$. Taking the coordinate $w_1=u_{\si 1}+1/2,\,w_2=u_{\si 2}+1,\,w_3=u_{\si 3}$,
 we have
\[\begin{split}
& {\hat\pi}^*g=w_3^6(\bar g_6+w_3),\\
&\bar g_6=-3456 w_1^3-2304 w_2^2+\text{(higher terms)},\\
& {\hat\pi}^*u_{\si 3}=(w_1-1/2) w_3,\, u_{\si 3}=w_3.
\end{split}
\]
We take the toric  modification with respect to the same $\Si_\rho^*$, $ {\hat\omega}_3: Y_3\to X$ and we get 
 \begin{eqnarray}
( {\hat\omega}_1^* {\hat\pi}^*x)=6E(P)_3+3E(T)_3+2E(S)_3+E_0+(u_{\si 1}=0)\,\,\text{in}\,\,Y_3.
\end{eqnarray}
Note that this is the same equality with (\ref{rho1}).
Taking the same toric modification $ {\hat\omega}_j: Y_j\to X$ at all $\rho_j$
and we take the union of these 6 toric modifications to obtain the resolution of 
all 6 singular points $ {\hat\omega}: Y\to X$ and let $ {\hat\Pi}:Y\to \mathbb C^3$ be the composition  $ {\hat\pi}\circ  {\hat\omega}$.
We observe that  the configuration of the exceptional divisors of  $V(g)$ by $\Pi: \widetilde V\to V=V(g)$,
$\Pi=\hat\Pi |_{\widetilde V}$ is the same with that of the resolution of
$V(f)$.  Note that
\[
(\Pi^*x)=E_0+(u_{\si 1}=0)+\sum_{i=1}^6 (6E(P)_i+3E(T)_i+2E(S)_i),
\]
Note also that the intersection number  $(u_{\si 1}=0)\cdot E_0$ is $6$. This follows from the observation that $\{u_{\si 1}=0\}\cap E_0$ corresponds to the 6 roots of
$\bar g(0,u_{\si 2})=0$.  Using the above equality
we get
\begin{eqnarray}\label{divisorMg}
E(P)_i^2=-1,\,E(T)_i^2=-2,\,E(S)_i^2=-3,\, E_0^2=-42.
\end{eqnarray}
These intersection numbers are the  same with that of the resolution of  $V(f)$.
We can also check the genus of $E_0$ is 4.
 Let $I_g$ be the $19\times 19$ intersection matrix of the exceptional divisors. 
 The intersection matrix $I_f$ and $I_g$ are the same. 
The calculation of $\zeta'(t)$ and $\zeta_{\rho_i}(t)$ is also  exactly same.
Thus we conclude
 \begin{Theorem}\label{main2}
 The functions$f$ and $g$ have the same zeta function:
 \[\begin{split}
& \zeta'(t)=(1-t^6)^{-9},\quad \zeta_{\rho_i}(t)=\frac{(1-t^{21})(1-t^{14})}{(1-t^{42})(1-t^7)}\\
 &\zeta(t)=\frac 1{(1-t^6)^9}\left(\frac{(1-t^{21})(1-t^{14})   }{ (1-t^{42})(1-t^7)} \right)^6
 \end{split}
 \]
 and the Milnor number is $137$. They   have  the same resolution graph with 19 exceptional divisors
 and their self-intersection numbers are given as
 $E(P)_i^2=-1, E(T)_i^2=-2, E(S)_i^2=-3$ and $E_0^2=-42$. The determinant of the  intersection matrix is $-6$.
 \end{Theorem}
 Note that $137=(6-1)^3+6\times 2$ and $5^3$ is the Milnor number of non-degenerate homogeneous  polynomial of degree $6$. 
 This observation is generalized in Theorem \ref{MilnorFormula} in \S 5. For the calculation of $\zeta'(t)$, 
 we use the equalities $\zeta^{(s)}(t)=(1-t^d)^{-21}$
 and $\zeta^{er}(t)=(1-t^6)^{12}$.
 
  \vspace{1cm}
 \begin{figure}[htb]  
\setlength{\unitlength}{1bp}
\begin{picture}(600,300)(-50,-30)
{\includegraphics[width=8cm, bb=0 0 595 842]{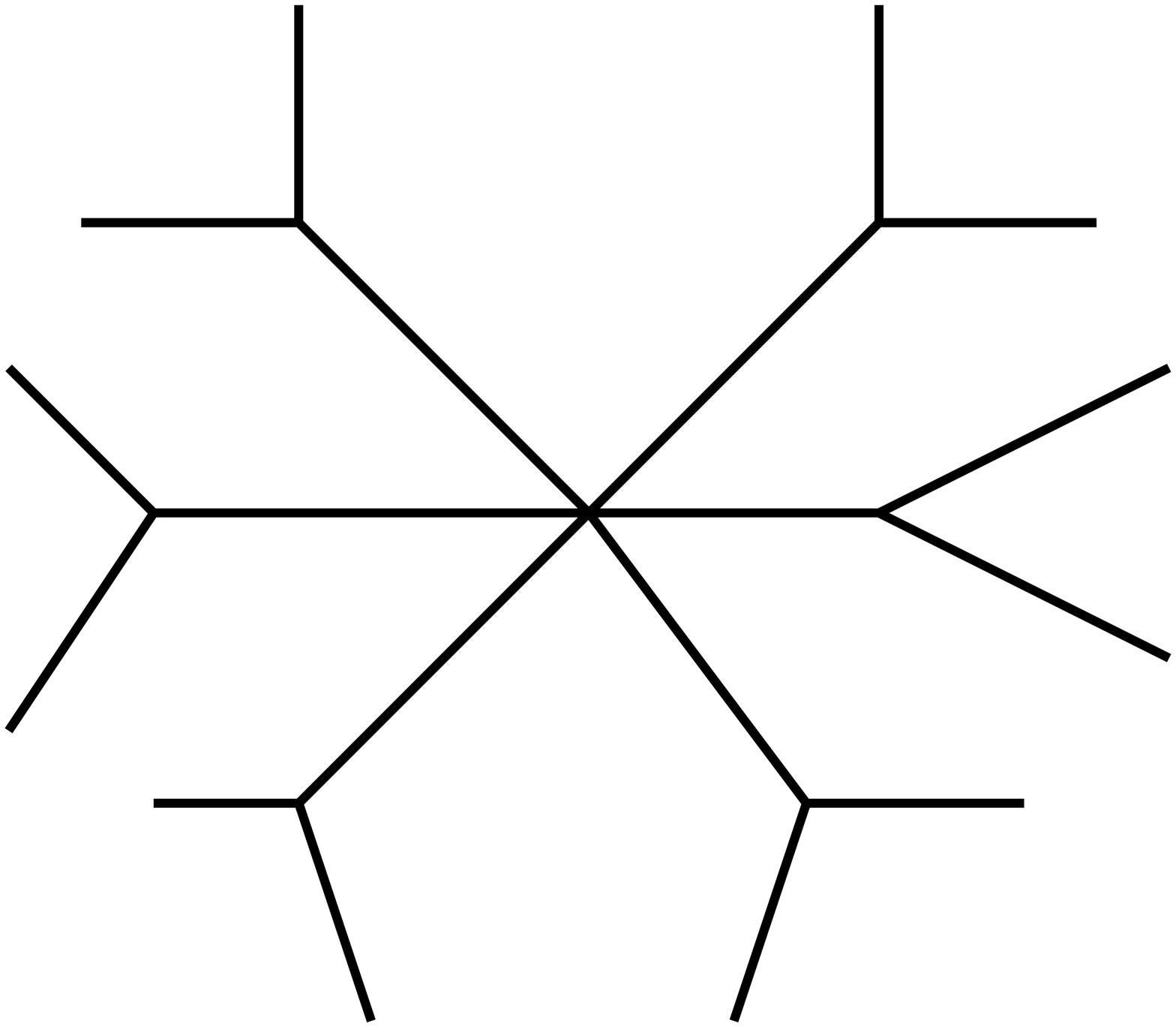}}
\put(-168, 218){\circle*{7}}\put(-168,260){\circle*{7}}
\put(-209,218){\circle*{7}}
\put(-160,255){$E(S)_4$}\put(-188,235){$\Ga_4$}\put(-230,225){$E(T)_4$}
\put(-113,162){\circle{12}}\put(-113,162){\circle*{7}}\put(-118,140){$E_0$}
\put(-70,140){$E(P)_2$}
\put(-200,140){$E(P)_5$}
\put(-70,105){\circle*{7}}\put(-30,105){\circle*{7}}
\put(-74,115){$ E(P)_1$}\put(-25,105){$E(T)_1$}
\put(-78,60){$E(S)_1$}\put(-60,80){$\Ga_1$}
\put(-85,67){\circle*{7}}
\put(-55,161){\circle*{7}}\put(-2,134){\circle*{7}}
\put(0,189){\circle*{7}}
\put(5,180){$E(T)_2$}\put(0,150){$\Ga_2$}
\put(5,130){$E(S)_2$}
\put(-57,216){\circle*{7}}\put(-17,216){\circle*{7}}
\put(-57,260){\circle*{7}}
\put(-10,220){$E(S)_3$}\put(-27,230){$\Ga_3$}\put(-50,260){$E(T)_3$}
\put(-90,220){$E(P)_3$}
\put(-160,220){$E(P)_4$}
\put(-196,107){\circle*{7}}\put(-167,105){\circle*{7}}
\put(-152,62){\circle*{7}}
\put(-230,100){$E(S)_6$}\put(-175,75){$\Ga_6$}\put(-140,60){$E(T)_6$}\put(-160,95){$E(P)_6$}
\put(-223,118){\circle*{7}}
\put(-223,190){\circle*{7}}\put(-196,160){\circle*{7}}
\put(-250,200){$E(S)_5$}\put(-228,140){$\Ga_5$}\put(-265,115){$E(T)_5$}
\end{picture}
\vspace{-3cm}
\caption{Resolution Graph}\label{Graph}
\end{figure}

 \begin{Remark}{\rm
 Consider a  $(p,q)$-torus curve $f_{pq}=0$ (\cite{Okapq}) and its isolation function $h$: 
 \[\begin{split}
& f_p:=x^p+y^p+z^p,\,f_q=x^q+y^q+z^q,\,\gcd(p,q)=1\\
&h=f_{pq}+z^{pq+1},\, f_{pq}=f_p^q+f_p^q.
\end{split} \]
 The calculation of the zeta function  is completely parallel for  $f$. We take first a blowing up $ {\hat\pi}: X\to \mathbb C^3$ and we take the chart
$(U,\mathbf u_\si=(u_{\si 1},u_{\si 2},u_{\si 3}))$ with $ {\hat\pi}(\mathbf u_\si)=(u_{\si 1}u_{\si 3},u_{\si 2}u_{\si 3},u_{\si 3})$. Denote the exceptional divisor  by $\hat E_0$ which is $\mathbb P^2$ and it 
 is defined by $u_{\si 3}=0$ in $U$. $E_0=\hat E_0\cap \widetilde V$ is the exceptional divisor of the restriction to $\widetilde V(h)$.
$E_0$ is equal to  the projective curve defined by $f_{pq}=0$ in $\mathbb P^2$ and it has $pq$ singularities $\rho_i,\,i=1,\dots, pq$ which are $(p,q)$-cusps: $x^p+y^q=0$.
 After one ordinary blowing up $\hat\pi:X\to \mathbb C^3$, 
 we have
 \[\begin{split}
 &\hat\pi^*h=u_{\si 3}^{d}(\bar f_q^p+\bar f_p^q+u_{\si 3})\\
 &\bar f_p=u_{\si 1}^p+u_{\si 2}^p+1,\,\bar f_q=u_{\si 1}^q+u_{\si 2}^q+1
 \end{split}\]
  and 
 the zeta function of $h$ can be computed in the exact same way  as follows.
 \[\begin{split}
 &\zeta'(t)=(1-t^ {pq})^{- {pq}^2+3 {pq}-3+ {pq}(p-1)(q-1)},\\
 &\zeta_{\rho_i}(t)=\frac{(1-t^{p( {pq}+1)})(1-t^{q( {pq}+1)})}{(1-t^{ {pq}( {pq}+1})(1-t^{ {pq}+1})},
 \, i=1,\dots, pq\\
 &\zeta(t)=\frac 1{(1-t^ {pq})^{ {pq}^2-3 {pq}+3- {pq}(p-1)(q-1)}}
 \left(\frac{(1-t^{p( {pq}+1)})(1-t^{q( {pq}+1)})}{(1-t^{ {pq}( {pq}+1)})(1-t^{ {pq}+1})}
 \right)^{ {pq}}.
 \end{split}
 \]
 and $\mu(h)=(pq-1)^3+pq(p-1)(q-1)$.
 }
 \end{Remark}
 
 \subsubsection{Computation of $H_1(M_f)$ and $H_1(M_g)$}
 Consider compact Riemann surfaces $E_1,\dots, E_\ell$  embedded in a complex manifold of dimension $2$
 which  intersect either transversely at a point or does not intersect and no three $E_i,\,E_i,\,E_k$  intersect.  Consider the  graph $\Ga$ with vertices $v_1,\dots, v_\ell$ which correspond to $E_1,\dots, E_\ell$. We assume that $\Ga$ is a tree graph.
 Let $s_{i j}$ be the intersection number $E_i\cdot E_j$ and let $s_{i i}$ be the self-intersection number of $E_i$.  Take a small tubular neighborhood $N(E_i)$ and we assume that they intersect transversely and  $N(E_i)\cap N(E_j)$ is diffeomorphic to $D^2\times D^2$ if $s_{i j}=1$.
 Put $N(\Ga)=\cup_{i=1}^\ell N(E_i)$ and put $M(\Ga)$ be the boundary of $N(\Ga)$. $M(\Ga)$ is a 3-manifold with corners. Assume the $E_j$ is $\mathbb P^1$ for each $j$. Taking suitable generators $g_j$ of $\pi_1(M(\Ga))$ represented by a fiber of $\partial N(E_j)$, the fundamental group is generated by $g_1,\dots, g_\ell$ with the relations
 \[[g_i,g_j]=e, \,\text{if}\, s_{ij}=1\,\,\text{and}\,\,
 g_1^{s_{1 j}}\dots g_\ell^{s_{\ell j}}\,\,\text{for}\, 1\le j\le \ell.
 \]
 This is shown in \cite{Mumford}.
 Now we compute the homology of $M_f$ and $M_g$ identifying them as the graph manifolds. 
 Take the tubular neighborhood $N(E(P)_i)$,  $N(E(T)_i)$ and $N(E(S)_i)$ for $1\le i\le 6$ and $N(E_0)$ sufficiently small and let $N=N(E_0)\cup_{i=1}^6(N(E(P)_i)\cup N(E(T)_i)\cup N(E(S)_i))$. 
 Let $N(\Ga_i)'$ be the boundary of the subgraph $\Ga_i$, cut off $D_i^2\times S^1$ at  the intersection of $N(P_i)$ and $N(E_0)$. Note that $\partial N(\Ga_i)'$ is $\partial D_i\times S^1$.
 We compute $H_1(\partial N(\Ga_i)')$ following the recipe of \cite{Mumford, Hirz}.
 Take  1-cycles $p_i, t_i, s_i$ represented by the respective fiber of the boundary of the tubular neighborhoods of 
 $\partial N(E(P)_i)$, $\partial N(E(T)_i)$, $\partial N(E(S)_i))$ and $e_0$ the fiber of $N(E_0)$.
 Then $H_1(N(\Ga_i)')$ is generated by $e_0,\,p_i,\,t_i,\,s_i$  which satisfies the relations:
 \begin{eqnarray}
 p_i-2t_i=0,\,p_i-3s_i=0,\,e_0-p_i+t_i+s_i=0, 1\le i\le 6.
 \end{eqnarray}
 See  page 10, \cite{Mumford}. Thus we can  solve this as 
\begin{eqnarray}\label{relation-i}
p_i=6e_0,\,t_i=3e_0,\,s_i=2e_0. 
\end{eqnarray}
Let $D_1,\dots, D_6$ be small disks on $E_0$ obtained as the intersection of $E_0$ and $N(E(P)_i)$,
$i=1,\dots, 6$.
Take two disks $D, D'$ such that $D\subsetneq D'$ and $D'\setminus D$ include 6 small disks 
$D_1,\dots, D_6$. The restriction of $N(E_0)$ over $E_0\setminus  D$ is trivial and 
diffeomorphic to 
$(E_0\setminus \Int( D))\times S^1$ and under the gluing with $D\times S^1$,
$\partial D\times \{*\}$ is homologous to $-42e_0$ where $e_0$ is represented by
$\{*\}\times S^1\subset (E_0\setminus \Int(D))\times S^1$.
Take the generators $a_1,b_1,\dots, a_4,b_4$ of $\pi_1(E_0)$ in $E_0\setminus D'$
so that $\partial {D'}\inv=[a_1,b_1]\dots [a_4,b_4]$. As the boundary of
the region
$D'\setminus (D\cup_{i=1}^6D_i)$
is $\partial D'-\partial D-\sum_{i=1}^6 \partial D_i$, we get the relation:
 \[
(R_0)\qquad  -(p_1+\dots+p_6)-\partial D=0\,\,\i.e.\,\, 6e_0=0\,\,\text{by (\ref{relation-i})}.
\]
where the boundary of disks are oriented counterclockwise and the homology class of 
$[a_1,b_1]\cdots [a_4,b_4]$ is zero. See Figure \ref{E0}. The relation $(R_0)$ is the same as the assertion on p. 10, in \cite{Mumford}

 \begin{figure}[htb]  
\setlength{\unitlength}{1bp}
\begin{picture}(600,300)(-50,-30)
{\includegraphics[width=8cm, bb=0 0 595 842]{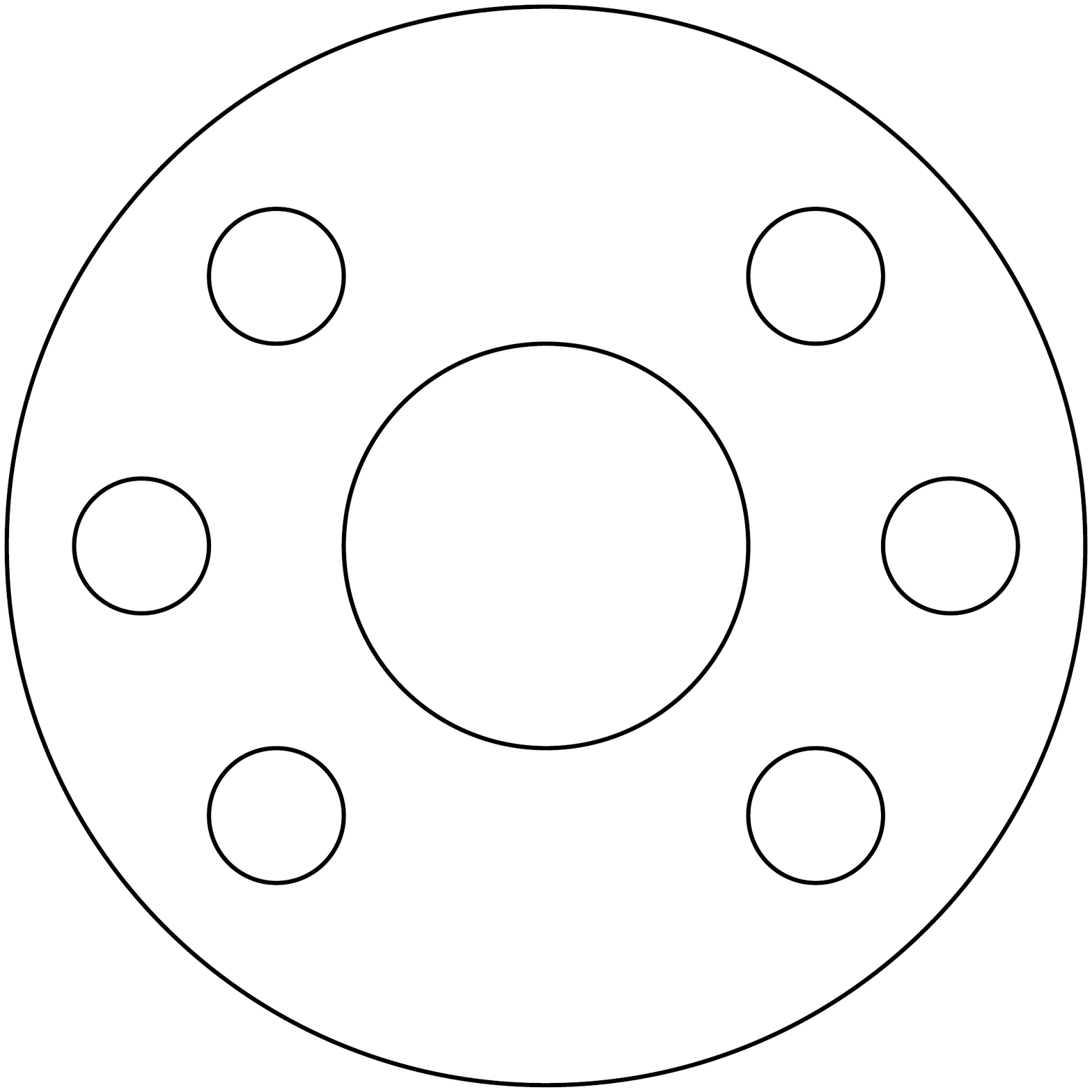}}
\put(-170, 216){$D_3$}
\put(-110,225){$D'$}
\put(-110, 158){$D$}
\put(-64,218){$D_2$}
\put(-35,158){$D_1$}
\put(-64,100){$D_6$}
\put(-198,158){$D_4$}\put(-171,100){$D_5$}
\end{picture}
\vspace{-3cm}
\caption{On $E_0$}\label{E0}
\end{figure}
Thus we have the following:
\begin{Theorem}\label{main3} The  links  of the surfaces $V(f)$ and $V(g)$ are diffeomorphic   and their irst homology group is given as:
 \[\begin{split}
&H_1(M_f)=H_1(M_g)=\mathbb Z^8\oplus \mathbb Z/6\mathbb Z.
\end{split}
\]
\end{Theorem}
The order of the torsion parts  come from the absolute values of the  determinant of the intersection matrix.
\begin{Problem}The homeomorphism of $M_f$ and $M_g$ comes from the graph manifold structure.
Does this homeomorphism   extend to a homeomorphism of the sphere $S_\eps^{5}$?
\end{Problem}
\section{Sift formula for Milnor number}
\subsection{Non-degenerate  polynomial }
Let $h(\mathbf y)=\sum_{i=1}^k a_i\mathbf y^{\nu_i}$ be a  polynomial of $m$ variables $y_1,\dots, y_m$.
The Newton polygon $\De(h)$ is defined by the convex hull of $\{\nu_1,\dots,\nu_k\}$ in $\mathbb R^m$.
We assume here that $a_i\ne 0$ for any $1\le i\le k$.
For a face $\Xi\subset \De(h)$ ($\Xi$ can be $\De(h)$ itself), $h_{\Xi}(\mathbf y)$ is defined by the sum $\sum_{\nu_i\in \Xi}a_i\mathbf y^{\nu_i}$. $h$ is called {\em  Newton non-degenerate}
as a polynomial if $h_\Xi:\mathbb C^{*m}\to \mathbb C$ has no
critical points for any $\Xi\subset\De(h)$. If $f\mathbf z)=\sum_\nu a_\nu\mathbf z^\nu$ is a Newton non-degenerate function germ,  any face function $f_P(\mathbf z)$ for a strictly positive weight vector  is non-degenerate as a polynomial.
\begin{Lemma}[Kouchnirenko \cite{Ko}, Oka \cite{Ok3}] Assume that $h(\mathbf y)$ is Newton non-degenerate  as a polynomial  and let
$V(h)^*=\{\mathbf y\in \mathbb C^{*m}\,|\, h(\mathbf y)=0\}$.
Then the Euler characteristic  is given as 
\[ \chi(V(h)^*)=(-1)^{m-1}m! {\Vol}_m \De(h).
\]
\end{Lemma}
In particular, if $\dim\,\De(h)<m$, $\chi(V(h)^*)=0$.
\begin{Corollary}\label{corollary}
Let $h(\mathbf y)$ be a   non-degenerate polynomial and consider 
 a polynomial of $m+1$ variables $\hat h(\mathbf y, w):=h(\mathbf y)+w$.
Then $\hat h$ is also   non-degenerate and  the  following equality holds.
\[\begin{split}
m!{\Vol}_m(\De(h))=(m+1)! {\Vol}_{m+1}(\De(\hat h))\,\,\text{and}\,\,
&\chi(V(\hat h)^*)=-\chi(V(h)^*).
\end{split}
\]
\end{Corollary}

\subsection{Shift of Milnor number}
Consider a convenient  homogeneous polynomial  $f_d(\mathbf z)$ of degree $d$ which defines a projective hypersurface $V$ with
$s$ isolated singular points $\rho_1,\dots, \rho_s$. We assume the singular points are in the projective chart $\{z_n\ne 0\}$.
 We assume that for each $\rho_i$, there exists a local coordinates $(U_i,\mathbf w)$ so that  the local equation of $V\cap U_{\rho_i}$ is a Newton non-degenerate function $f_{i}(\mathbf w)$. Put $\mu_i$ the Milnor number of $f_i$ at $\rho_i$ and put $\mu_{tot}:=\sum_{i=1}^s \mu_i$.
Consider the modified function $f(\mathbf z)=f_d(\mathbf z)+z_n^{d+1}$ which makes $V(f)$ has an isolated singularity at  the origin.
Then we  have:
\begin{Theorem}[Shift formula of Milnor number] \label{MilnorFormula}Assume that $f_d(\mathbf z)$ be as above. Then
$f(\mathbf z)=f_d(\mathbf z)+z_n^{d+1}$ is an almost Newton non-degenerate function and the Milnor number $\mu(f)$ of $f$ is given as
$\mu(f)=(d-1)^n+\mu_{tot}$.
\end{Theorem}
\begin{proof}
Note that $(d-1)^n$ is the Milnor number of a homogeneous polynomial of degree $d$.
The zeta function $\zeta(t)$ of $f$ is given  by Theorem \ref{main1} as
\begin{eqnarray}\label{zeta-formula}
\zeta(t)=\zeta_d(t) (1-t^d)^{(-1)^{n-1}\mu_{tot}} \prod_{i=1}^s \zeta_{\rho_i}(t)
\end{eqnarray}
where $\zeta_d(t)$ is the zeta function of a convenient homogeneous polynomial of degree d with an isolated singularity at the origin of $\mathbb C^n$ and $\zeta_{\rho_i}(t)$ is the zeta function of the $\hat\pi^*f$ at $\rho_i$. 
First we take an ordinary blowing up $\hat\pi:X\to \mathbb C^n$ at the origin and take  the chart $U_n$ with coordinates $(u_1,\dots, u_n)$ so that $\pi(\mathbf u)=\mathbf z$
with
$z_i=u_iu_n,\,i\le n-1$ and $z_n=u_n$.  Put $\mathbf u'=(u_1,\dots, u_{n-1})$. The exceptional divisor is defined by $u_n=0$. Consider the pull-back of the functions:
\begin{eqnarray}
 \hat\pi^*f_d(\mathbf u)&=&u_n^d\widetilde {f_d}(\mathbf u'),\,\widetilde {f_d}(\mathbf u')=f_d(\mathbf u',1),\label{lift1}\\
 \hat\pi^*f(\mathbf u)&=&u_n^d\widetilde f(\mathbf u),\, \widetilde f(\mathbf u):=\widetilde {f_d}(\mathbf u')+u_n.\label{lift2}
\end{eqnarray}
For simplicity, we use the notations $\hat f_d(\mathbf u)$ and $\hat f(\mathbf u)$ for $\hat\pi^*f_d(\mathbf u)$ and 
$\hat\pi^*f(\mathbf u)$.
Note that $\mathbf u'$ can be considered as  the projective coordinates of $\{u_n\ne 0\}\cap \mathbb P^{n-1}$
and $\widetilde {f_d}(\mathbf u')$ is the defining polynomial of $V(f_d)\cap \subset \mathbb P^{n-1}$ and $\widetilde f(\mathbf u)$ is the defining polynomial of the strict transform of $V(f)$. Take a singular point $\rho_i$ and 
choose a local coordinates $\mathbf w'=(w_1,\dots, w_{n-1})$ of the exceptional divisor $\{u_n=0\}\cong \mathbb C^{n-1}$ at $\rho_i$ so that
$\widetilde{f_d}(\mathbf w')$ is non-degenerate with respect to this coordinates.
The zeta function $\zeta_i(t)$ of $\widetilde{ f_d}(\mathbf w')$ at $\rho_i$ is given  by  Varchenko formula as 
\[\begin{split}
&\zeta_i(t)=\prod_{I}\zeta_I(t)\\
&\zeta_I(t)=\prod_{Q\in \mathcal P_I(\widetilde {f_d})}(1-t^{d(Q,{\widetilde {f_d}}^I)})^{-\chi(Q)},\\
&\chi(Q)=(-1)^{|I|}|I|!{\Vol}_{|I|}{\Cone}(\De(Q))/d(Q,{\widetilde{f_d}}^I)
\end{split}
\]
where
$\mathcal P_I(\widetilde{f_d})$ is the primitive weight vectors corresponding to the maximal faces of $\Ga(\widetilde{f_d}^I(\mathbf w'))$
with $I\subset \{1,\dots, n-1\}$. 
As a corollary,
we have
\begin{eqnarray}\label{localMilnor}
-1+(-1)^{n-1}\mu_i&=&\deg\zeta_i(t)\\
&=&\sum_{I}\sum_{Q\in \mathcal P_I(\widetilde{f_d})}(-1)^{|I|} |I|!{\Vol}_{|I|}{\Cone}(\De(Q,\widetilde {f_d}^I))\notag\\
&=&\sum_{I}\sum_{Q\in \mathcal P_I({\widetilde{f_d}})}-d(Q,\widetilde{f_d}^I)\chi(Q).\notag
\end{eqnarray}
Similarly we define $\mathcal P_I(\hat f)$ the set of weight vectors corresponding to 
 maximal faces of $\Ga({\hat f}^I(\mathbf w))$ where $\mathbf w=(\mathbf w',w_n)$ and $w_n=u_n$.
Observe that the dual Newton diagram $\Ga^*(\hat f(\mathbf w))$ is equal to  the dual Newton diagram of the reduced function  $\widetilde {f}(\mathbf w)$, as $\hat f$ is pseudo-convenient.
Also observe that   $\hat f^I$ is not identically zero if and only if $n\in I$.
It is clear that 
$\hat f(\mathbf w)$ is non-degenerate with  respect to this coordinates as $\widetilde {f_d}(\mathbf w')$
is non-degenerate by the assumption. Thus $f(\mathbf z)$ is an almost non-degenerate function.
 $\hat f(\mathbf w)$ is also  pseudo-convenient. Thus we can resolve the singularity by a toric modification which is biholomorphic outside of the origin. This shows, in particular, $f$ has an isolated singularity at the origin.
For  a strictly positive vertex $P={}^t(p_{1},\dots, p_{n-1})\in \mathcal P(\widetilde{f_d}(\mathbf w'))$ of $\Ga^*(\widetilde{f_d})$, we put $\hat P={}^t(p_{1},\dots, p_{{n-1}},p_n)$ with $p_n= d(P,\tilde {f_d})$. For  $I\subset \{1,\dots,n-1\}$ and we  put $\hat I=I\cup\{n\}$.  $\hat P$ is a weight vector of $\mathbf w$ and by the definition of $\hat P$, $\hat f_{\hat P}(\mathbf w)=w_n^d(\widetilde{f_d}_{P}(\mathbf w')+w_n)$.
Define  a weight vector $\hat Q$ 
 in the similar way for 
$Q\in \mathcal P_I(\widetilde {f_d})$.
Then for $Q\in \mathcal P_I(\widetilde {f_d}^I)$, 
\begin{eqnarray}
 {\hat f}_{\hat Q}^{\hat I}(\mathbf w)&=&({\widetilde {f_d}}_Q^I(\mathbf w')+w_n) w_n^d,\,\,
d(\hat Q,{\hat f}^{\hat I})=d(Q,\widetilde {f_d}^I)(1+d).
\end{eqnarray}
The zeta function $\zeta_{\rho_i}(t)$ of $\hat f(\mathbf w)$ at $\rho_i$ is given as
\begin{eqnarray}
\zeta_{\rho_i}(t)&=(1-t^{d+1})^{-1}\prod_{Q\in \mathcal P_I(\widetilde {f_d})}(1-t^{d(\hat Q,\hat f^{\hat I})})^{-\chi(\hat Q)}.
\end{eqnarray}
Here the term $(1-t^{d+1})^{-1}$ comes from $ \{n\}$, the single monomial $w_n^{d+1}$ which is the only case, not corresponding to any factor of $\zeta_i(t)$.
We assert:
\begin{Assertion}\label{assertion}
 $\chi(Q)=-\chi(\hat Q)$ for $Q\in \mathcal P_I(\widetilde{f_d})$.
\end{Assertion}
As the argument is completely parallel for any $I$, we assume that $I=\{1,\dots, n-1\}$.
We consider a regular simplicial cone subdivision $\Si_i^*$ of $\Ga^*(\widetilde {f_d})$ and put  $\hat \Si_i^*$
the join of $\Si_i^*$ and $e_n$. It gives an admissible regular simplicial subdivision of $\Ga^*(\hat f)$.
 Take the corresponding toric modification
$\omega:X\to \mathbb C^{n-1}$ and $\hat\omega:Y\to \mathbb C^n$ and take a maximal simplex
$\si={\Cone}(P_1,\dots, P_{n-1})$ such that $Q=P_1$ and put $\hat \si={\Cone}(\hat P_1,\dots, \hat P_{n-1},e_n)$. 
In the coordinate chart $\mathbb C_{\si}^{n-1}$  and $\mathbb C_{\hat \si}^n$ with coordinates
$\mathbf u_\si'=(u_{\si,1},\dots, u_{\si,n-1})$ and $\mathbf u_{\hat \si}=(\mathbf u_\si',u_{\si n})$, we have
\[\begin{split}
\omega^*{\widetilde{f_d}}_{ Q}(\mathbf u_\si')
&=(\prod_{i=1}^{n-1}u_{\si i}^{d(P_i,\widetilde {f_d})})
\overline{f_d}_Q(  \mathbf u_\si''),\,\mathbf u_\si''=(u_{\si 2},\dots, u_{\si,n-1}),\\
{\hat\omega^*{\hat f}}_{\hat Q}(\mathbf u_{\hat \si})
&=(\prod_{i=1}^{n-1}u_{\si i}^{d(P_i,\widetilde {f_d})})u_{\si n}^{d}
(\overline{f_d}_Q( \mathbf u_\si'')+u_{\si n}).
\end{split}
\]
Here ${\widetilde{f_d}}_Q(\mathbf u_{\si}')$ and $\widetilde{f_d}_Q(\mathbf u_\si')+u_{\si n}$ are the defining polynomials of $E(Q)$ and $E(\hat Q)$. The polynomial  ${\widetilde {f_d}}_Q$ is  non-degenerate as a polynomial by the non-degeneracy assumption of $\widetilde {f_d}(\mathbf w')$.
Thus we have
\begin{eqnarray*}
\chi(Q)&=&(-1)^{n-2} (n-1)!{\Vol}_{n-1}{\Cone}(\De(Q,\widetilde {f_d}))/d(Q,\widetilde{f_d})\\
&=&(-1)^{n-2}(n-2)!{\Vol}_{n-2}\De(\overline{f_d}_Q(\mathbf u_\si'')),\\
\chi(\hat Q)&=&(-1)^{n-1}n!{\Vol}_{\hat Q}{\Cone}(\De(\hat Q,\hat f))/d(\hat Q,\hat f)\\
&=&(-1)^{n-1}(n-1)!{\Vol}_{n-1}\De(\overline{f_d}_Q( \mathbf u_\si'')+u_{\si n})\\
&=&-\chi(Q)\,\,\text{by Corollary \ref{corollary}}.
\end{eqnarray*}
This proves the Assertion \ref{assertion}.
Now we are ready to show the assertion for Milnor number.
By the above argument, we have
\[\begin{split}
-1+(-1)^{n-1}\mu_i&=\sum_{I}\sum_{Q\in \mathcal P_I}-d(Q,\widetilde {f_d})\chi(Q),\\
\deg\, \zeta_{\rho_i}(t)&=-(1+d)+\sum_I\sum_{Q\in \mathcal P_I}-d(Q,\widetilde{f_d})(1+d)\chi(\hat Q)\\
&=-(1+d)+\sum_I\sum_{Q\in \mathcal P_I}d(Q,\widetilde{f_d})(1+d)\chi( Q)\\
&=-(1+d)-(d+1)(-1+(-1)^{n-1}\mu_i)\\
&=(-1)^n(1+d)\mu_i
\end{split}
\]
Thus
we get 
\[\begin{split}
-1+(-1)^{n}\mu(f)&=\deg\,\zeta(t)\\
&=-1+(-1)^n\mu_d+(-1)^{n-1}d\mu_{tot}+\sum_{i=1}^s\deg\,\zeta_{\rho_i}(t)\\
&=-1+(-1)^n(\mu_d+\mu_{tot}).
\end{split}
\] 
This comples the proof of Theorem \ref{localMilnor}.
\end{proof}

Before closing the paper, 
the author  thanks Professor Kimihiko Motegi for precious information and  fruitful discussions about the topology of the link manifolds.
\def\cprime{$'$} \def\cprime{$'$} \def\cprime{$'$} \def\cprime{$'$}
  \def\cprime{$'$} \def\cprime{$'$} \def\cprime{$'$} \def\cprime{$'$}

\end{document}